\documentclass[a4paper,11pt,twoside,reqno]{amsart}

\usepackage[utf8]{inputenc}
\usepackage[plainpages=false,pdfpagelabels=true]{hyperref}
\usepackage{amssymb,amsthm}
\usepackage[margin=1in]{geometry}
\usepackage[dvipsnames]{xcolor} 

\newtheorem{Theorem}{Theorem}[section]

\newtheorem{Lem}[Theorem]{Lemma}
\newtheorem{Thm}[Theorem]{Theorem}

\theoremstyle{definition}

\newtheorem{Bem}[Theorem]{Remark}

\newcommand{\tr}{\operatorname{Tr}}

\newcommand{\dv}{\text{ }dV}

\allowdisplaybreaks[1]

\renewcommand{\epsilon}{\varepsilon}

\newcommand{\R}{\ensuremath{\mathbb{R}}}

\newcommand{\s}{\mathbb{S}}

\numberwithin{equation}{section}

\title[Unique continuation properties for polyharmonic maps]{Unique continuation properties for polyharmonic maps between Riemannian manifolds}

\author{Volker Branding}
\address{University of Vienna, Faculty of Mathematics\\
Oskar-Morgenstern-Platz 1, 1090 Vienna, Austria}
\email{volker.branding@univie.ac.at}

\author{Stefano Montaldo}
\address{Universit\`a degli Studi di Cagliari\\
Dipartimento di Matematica e Informatica\\
Via Ospedale 72\\
09124 Cagliari, Italy}
\email{montaldo@unica.it}

\author{Cezar Oniciuc}
\address{Faculty of Mathematics\\
Al.I. Cuza University of Iasi, Bd. Carol I, 11, 700506 Iasi, Romania}
\email{oniciucc@uaic.ro}

\author{Andrea Ratto}
\address{Universit\`a degli Studi di Cagliari\\
Dipartimento di Matematica e Informatica\\
Via Ospedale 72\\
09124 Cagliari, Italy}
\email{rattoa@unica.it}

\date{\today}



\begin{document}

\begin{abstract}\textit{Polyharmonic maps} of order $k$ (briefly, $k$-\textit{harmonic maps}) are a natural generalization of harmonic and biharmonic maps. These maps are defined as the critical points of suitable higher order functionals which extend the classical \textit{energy functional}  for maps between Riemannian manifolds. The main aim of this paper is to investigate the so-called \textit{unique continuation principle}. 
{More precisely, assuming that the domain is connected, we shall prove the following extensions of results known in the harmonic and biharmonic case: (i) if a $k$-harmonic map is harmonic on an open subset, then it is harmonic everywhere; (ii) if two $k$-harmonic maps agree on a open subset, then they agree everywhere; (iii) if, for a $k$-harmonic map to the $n$-dimensional sphere, an open subset of the domain is mapped into the equator, then all the domain is mapped into the equator. 
}
\end{abstract}
\maketitle

\section{Introduction and results}
\textit{Harmonic maps} are among the most studied geometric variational problems in differential geometry.
The geometric setup is the following. We consider a map \(\varphi\colon M\to N\) between two Riemannian manifolds \((M^m,g)\) and \((N^n,h)\).
Then the \textit{energy} of \(\varphi\) is defined by
\begin{align}
\label{1-energy}
E(\varphi)(=E_1(\varphi))= \frac{1}{2}\int_M|d\varphi|^2\dv\, .
\end{align}
Its critical points are governed by the vanishing of the so-called
\emph{tension field} $\tau(\varphi)$, that is
\begin{align}
\label{harmonic-map-equation}
0=\tau(\varphi):=\tr_g \nabla d\varphi= \sum_{j=1}^m  \nabla d\varphi (e_j,e_j)\, ,
\end{align}
where \(\{e_j\},j=1,\ldots,m=\dim M\), is a local orthonormal frame field tangent to \(M\) and the {\em second fundamental form} $\nabla d \varphi$ is defined by
\[
\nabla d \varphi(X,Y)=\bar \nabla_{X} d\varphi(Y)-d \varphi\left (\nabla_X Y \right )\, ,
\]
\(\bar \nabla\) being  the connection on the vector bundle \(\varphi^{-1}TN\). Harmonic maps are precisely the solutions of equation \eqref{harmonic-map-equation}.
We observe that the harmonicity equation is a second order semilinear elliptic system.
Since this system of equations is of second order, powerful tools such as the maximum principle
help to obtain a deep understanding of both analytic and geometric properties of harmonic maps (we refer to the classical surveys of Eells and Lemaire \cite{MR703510, MR1363513} for an introduction and background on this topic).

Another geometric variational problem which received growing attention in the recent years is that of the so-called \emph{biharmonic maps}.
These maps are characterized as the critical points of the \emph{bienergy} for maps between two Riemannian manifolds,
which is given by
\begin{align}
\label{2-energy}
E_2(\varphi)=\frac{1}{2} \int_M|\tau(\varphi)|^2\dv \, .
\end{align}
Here, the Euler-Lagrange equation is a fourth order semilinear elliptic system and is expressed by means of the vanishing
of the \emph{bitension field} $\tau_2(\varphi)$, that is
\begin{align}
\label{biharmonic-map-equation}
0=\tau_2(\varphi):=\bar\Delta\tau(\varphi)+\sum_{j=1}^mR^N(d\varphi(e_j),\tau(\varphi))d\varphi(e_j)\, ,
\end{align}
where \(\bar\Delta\) is the so-called \textit{rough Laplacian}, i.e., the connection Laplacian on \(\varphi^{-1}TN\). For background and research on biharmonic maps we refer to \cite{SMCO} and the recent book \cite{chenoubook}.

In contrast to the harmonic map equation, \eqref{biharmonic-map-equation} is of fourth order, a fact which entails significant additional technical difficulties. For instance, classical tools such as the maximum principle are no longer applicable (for instance, see \cite{MR3045700, MR3871569}).

There exist different systematic approaches which generalize the notions of harmonic and biharmonic
maps to energy functionals that contain derivatives of higher order.

In this paper we shall focus on the following $k$-order versions of the energy functional:
if $k=2s$, $s \geq 1$,
\begin{align}
\label{2s-energia}
E_{2s}(\varphi)
=&\frac{1}{2}\int_M \langle\bar{\Delta}^{s-1}\tau(\varphi),\bar{\Delta}^{s-1}\tau(\varphi)\rangle\dv\, .
\end{align}
In the case that $k=2s+1$:
\begin{align}
\label{2s+1-energia}
E_{2s+1}(\varphi)
=&\frac{1}{2}\int_M\langle\bar\nabla\bar{\Delta}^{s-1}\tau(\varphi),\bar\nabla \bar{\Delta}^{s-1}\tau(\varphi)\rangle \dv\, .
\end{align}
A \textit{polyharmonic map of order} $k$ (briefly, a $k$-\textit{harmonic map}) is a critical point of the $k$-energy functional $E_k(\varphi)$. 

{The functionals \eqref{2s-energia}, \eqref{2s+1-energia} are probably the simplest higher order version of the classical energy functional in a Riemannian geometric setting. These functionals were first studied systematically in an interesting series of papers by Maeta. Particularly, he established their main variational equations and properties, and proved some basic characterizations of proper (i.e. non-harmonic) triharmonic submanifolds into a sphere (see \cite{MR2869168, MR2911957, MR3007953}). 
More recently, these $k$-energy functionals} have been intensively studied. For instance, the stress-energy tensor for polyharmonic maps was recently calculated in \cite{MR4007262}.
Vanishing results for polyharmonic maps into Euclidean spaces have been obtained in \cite{MR3314128,MR3809656}
and for arbitrary targets in \cite{br2019structure}.
Moreover, proper $k$-harmonic immersions into spheres, ellipsoids and rotation hypersurfaces were constructed in \cite{MOR-arxiv,MR3711937, MR3790367}.
Some other results on triharmonic maps were achieved in \cite{MR3403738,MR3371364}.

The Euler-Lagrange equations of \eqref{2s-energia}, \eqref{2s+1-energia} were calculated in \cite{MR2869168, MR3007953} and can be described as follows (note that we set \(\bar\Delta^{-1}=0\)):
\begin{enumerate}
 \item The critical points of \eqref{2s-energia} are those which satisfy
\begin{equation}
\label{tension-2s}
\begin{split}
0=\tau_{2s}(\varphi):=&\bar\Delta^{2s-1}\tau(\varphi)-R^N(\bar\Delta^{2s-2}\tau(\varphi),d\varphi(e_j))d\varphi(e_j) \\
&-\sum_{\ell=1}^{s-1}\bigg(R^N(\bar\nabla_{e_j}\bar\Delta^{s+\ell-2}\tau(\varphi),\bar\Delta^{s-\ell-1}\tau(\varphi))d\varphi(e_j) \\
&\hspace{1cm}-R^N(\bar\Delta^{s+\ell-2}\tau(\varphi),\bar\nabla_{e_j}\bar\Delta^{s-\ell-1}\tau(\varphi))d\varphi(e_j)
\bigg)\, .
\end{split}
\end{equation}
\item The critical points of \eqref{2s+1-energia} are determined by
\begin{equation}
\label{tension-2s+1}
\begin{split}
0=\tau_{2s+1}(\varphi):=&\bar\Delta^{2s}\tau(\varphi)-R^N(\bar\Delta^{2s-1}\tau(\varphi),d\varphi(e_j))d\varphi(e_j)\\
&-\sum_{\ell=1}^{s-1}\bigg(R^N(\bar\nabla_{e_j}\bar\Delta^{s+\ell-1}\tau(\varphi),\bar\Delta^{s-\ell-1}\tau(\varphi))d\varphi(e_j) \\
&-R^N(\bar\Delta^{s+\ell-1}\tau(\varphi),\bar\nabla_{e_j}\bar\Delta^{s-\ell-1}\tau(\varphi))d\varphi(e_j)
\bigg) \\
&-R^N(\bar\nabla_{e_j}\bar\Delta^{s-1}\tau(\varphi),\bar\Delta^{s-1}\tau(\varphi))d\varphi(e_j)\, .
\end{split}
\end{equation}
\end{enumerate}
In this article we shall focus on one specific analytic aspect of solutions of the polyharmonic map equation, namely, the so-called \emph{unique continuation principle}. We refer to the work of Kazdan \cite{MR948075} for an introduction to this topic in a geometric setting. In particular, Kazdan exhibited an artificial counterexample, but also stated that it is reasonable to expect that the unique continuation property should hold in all geometrically meaningful situations. Indeed, for harmonic maps  unique continuation properties  were proven by Sampson in \cite{MR510549}
and recently generalized to biharmonic maps in \cite{MR3990379}.

{In this paper, all manifolds are assumed to be \textit{connected} and we shall work with \textit{smooth} objects only.}
Our first result is the following:
\begin{Thm}\label{Th-harmonic-everywhere}
Let $\varphi \colon M \to N$ be a $k$-harmonic map, $k \geq 3$. If $\varphi$ is harmonic on an open set $U$, then $\varphi$ is harmonic everywhere.
\end{Thm}
Next, we turn our attention to another, technically more demanding question concerning the unique continuation.

\textit{Let $\varphi,\,\tilde \varphi \colon M \to N$ be two $k$-harmonic maps, $k \geq 3$. If the two maps agree on an open subset of $M$, do they coincide everywhere?}

Our main result is to prove that the answer is affirmative. Indeed,
\begin{Theorem}
\label{main-theorem}
Let \(\varphi,\tilde\varphi\colon M\to N\) be two $k$-harmonic maps, $k \geq 3$. If they agree on an open subset $U$ of $M$, then they are identical.
\end{Theorem}

Moreover, we will also give a geometric application of Theorem \ref{main-theorem}
extending corresponding results for harmonic \cite[Theorem 6]{MR510549} and biharmonic maps \cite[Theorem 1.6]{MR3990379}.

\begin{Theorem}
\label{theorem-totally-geodesic-sphere}
Let \(\varphi\colon M\to \mathbb{S}^n\) be a \(k\)-harmonic map.
If an open subset of \(M\) is mapped into the equator \(\mathbb{S}^{n-1}\),
then all of \(M\) is mapped into \(\mathbb{S}^{n-1}\).
\end{Theorem}

The Euler-Lagrange system of equations which defines a $k$-harmonic map is elliptic of order $2k$ and has a rather complicated expression which depends on the Riemannian curvature tensor field of the target. Therefore, although all our proofs are based just on the application of the classical Aronszajn's unique continuation principle for second order elliptic operators, the technical steps that we shall have to carry out  are quite demanding and require a delicate use of suitable new variables. 

For the sake of completeness, we point out that another interesting generalization of both harmonic and biharmonic maps can be obtained
by studying the critical points of the following higher order energies:
\begin{align}
\label{k-energy-es}
E^{ES}_k(\varphi)=\frac{1}{2} \int_M|(d+d^\ast)^k\varphi|^2\dv\, ,\qquad k=1,2,\ldots
\end{align}
The study of these functionals was proposed by Eells and Sampson in 1965 (see \cite{MR0216519}) and, later, by
Eells and Lemaire in 1983 \cite[p.77, Problem (8.7)]{MR703510}.
A rigorous mathematical investigation of \eqref{k-energy-es}
has recently been initiated by the authors in \cite{MR4106647} and was further developed in \cite{br2020ES,MOR-arxiv}. We point out that, in general, the functional \eqref{k-energy-es} coincides with $E_k(\varphi)$ introduced above only when $k=1,2,3$.

In Section \ref{section-es4} we will show that Theorems \ref{Th-harmonic-everywhere} and \ref{main-theorem} also hold
for the critical points of \eqref{k-energy-es} in the case \(k=4\), which is the only case where the Euler-Lagrange equations are explicitly computed (see \cite{MR4106647}). However, since the technical difficulties are huge, we have preferred not to investigate other possible extensions of the unique continuation principle.

Throughout this article we shall use the following sign conventions and notations.
The Riemannian curvature
tensor field on a manifold $N$ is \(R^N(X,Y)Z=[\nabla_X,\nabla_Y]Z-\nabla_{[X,Y]}Z\) and, when the context is clear, we shall simply write $R$ instead of $R^N$. As for the rough Laplacian on \(\varphi^{-1}TN\), we shall use
\(\bar\Delta:=-\tr_g(\bar\nabla\bar\nabla-\bar\nabla_\nabla)\).
Similarly, the sign of the Laplace operator $\Delta$ on functions is such that $\Delta f= -f''$ on $\R$.

In general, we will use the same symbol $\langle\cdot,\cdot\rangle$ to indicate the Riemannian metrics on various vector bundles. We also note that on \(0\)-forms, that is on sections, \(d = \nabla\).

When the range is not explicitly specified we will use Latin indices \(i,j,k\) for indices on the domain ranging from \(1\) to \(m\)
and Greek indices \(\alpha,\beta,\gamma\) for indices on the target which take values between \(1\) and \(n\).
When the range of the indices is from $1$ to $q$, for some positive integer $q$, we will often denote them by \(a,b,c\).

We will use the Einstein summation convention, i.e., we will sum over repeated indices in the diagonal position. 

Most of our computations will be carried out in local charts and our convention concerning the indices of the sectional curvature tensor field is 
\begin{equation}\label{Convention-R}
R\left (\frac{\partial}{\partial y^\beta},\frac{\partial}{\partial y^\gamma} \right ) \frac{\partial}{\partial y^\delta}=R^\alpha_{\delta \beta \gamma}\, \frac{\partial}{\partial y^\alpha} \,.
\end{equation}
\section{Proof of the main results}

We recall the following classical result due to Aronszajn \cite[p.248]{MR0092067}:

\begin{Theorem}
\label{aro-theorem}
Let \(A\) be a second-order linear elliptic differential operator of class \(C^{\infty}\) defined on an open subset \(D\) of \(\R^m\).
Let \(u=(u^1,\ldots,u^q)\) be a function on \(D\) satisfying the inequality
\begin{equation}
\label{aro-voraus}
\left|Au^a\right|\leq C \left(\sum_{b,i}\left|\frac{\partial u^{b}}{\partial x^i}\right|+\sum_{b}\left|u^{b}\right|\right)
\end{equation}
for some $C>0$. If \(u=0\) in an open subset of $D$, then \(u=0\) throughout \(D\).
\end{Theorem}
\begin{Bem} In the literature, some unique continuation results for higher order elliptic equations are available (for instance, see \cite{MR0113030}). However, the great generality of the elliptic operator $A$ in \eqref{aro-voraus} persuaded us that Theorem \ref{aro-theorem} is the most effective available tool to achieve unique continuation in the context of polyharmonic maps. Moreover, as pointed out in \cite{MR510549}, also the strong version of the unique continuation principle holds, i.e., the conclusion of Theorem \ref{aro-theorem} is still true if $u=0$ to infinite order at some point. Therefore, as in \cite{MR510549}, both Theorems \ref{Th-harmonic-everywhere} and \ref{main-theorem} admit a strong formulation and the proof is the same.
\end{Bem}

In order to prove our results we need to write down a suitable local expression for the equations for polyharmonic maps \eqref{tension-2s}, \eqref{tension-2s+1}.

To this end, let us choose a local chart \((U,x^i)\) on \(M\) and a local chart \((V,y^\alpha)\) on \(N\) such that \(\varphi(U)\subset V\). To simplify the notation, we shall denote by $\varphi$ the expression of $\varphi$ in the two local charts.

It is well-known that, in local coordinates, the tension field is given by
\begin{align}
\label{tension-local}
\tau^\alpha(\varphi)=-\Delta\varphi^\alpha
+g^{ij}\, \Gamma_{\vartheta\beta}^\alpha\, \frac{\partial\varphi^\vartheta}{\partial x^i} \frac{\partial\varphi^\beta}{\partial x^j}\, ,\qquad 1\leq\alpha\leq n\, ,
\end{align}
where  \(\Gamma_{\vartheta\beta}^\alpha\) represent the Christoffel symbols of the manifold \(N\). We also recall that the Laplace-Beltrami operator $\Delta$ acts locally on a function $f\colon U \to \R$ as follows:
\[
-\Delta f = g^{ij}\, \frac{\partial^2 f}{\partial x^i \partial x^j}-g^{ij}\, \Gamma_{ij}^k\, \frac{\partial f}{\partial x^k}\, ,
\]
where here \(\Gamma_{ij}^k\) are the Christoffel symbols of the Riemannian manifold \(M\). Moreover, if 
\[
\sigma = \sigma^\alpha \frac{\partial}{\partial y^\alpha}
\]
is a section of $\varphi^{-1}TN$, then (see \cite[Lemma~1.1]{Ou}) 
\begin{equation}\label{eq:laplacian-on-sections}
\begin{split}
\bar{\Delta} \sigma =&\Big\{ \Delta\sigma^{\alpha}  -2\, g^{ij}\, \frac{\partial \sigma^\vartheta}{\partial x^j}\, \varphi_i^\beta\, \Gamma_{\beta \vartheta}^\alpha \\
& +\sigma^\vartheta  \left [ (\Delta \varphi^\beta)\, \Gamma_{\beta \vartheta}^\alpha - g^{ij}\, \varphi_j^\beta\, \varphi_i^\omega \; S^{\alpha}_{\beta \omega \vartheta}\right]\Big\} \frac{\partial}{\partial y^\alpha}\,,
\end{split}
\end{equation}
where 
\begin{equation}\label{eq:definitionofS}
2S^{\alpha}_{\beta \omega \vartheta} :=\frac{\partial \Gamma_{\beta \vartheta}^\alpha}{\partial y^\omega} + \Gamma_{\beta \vartheta}^\gamma\, \Gamma_{\omega \gamma}^\alpha+
\frac{\partial \Gamma_{\omega \vartheta}^\alpha}{\partial y^\beta} + \Gamma_{\omega \vartheta}^\gamma\, \Gamma_{\beta \gamma}^\alpha\, 
\end{equation}
and, for simplicity, $\varphi_i^\beta=\partial \varphi^\beta / \partial x^i$. 

\begin{proof}[Proof of Theorem \ref{Th-harmonic-everywhere}] 
Since the proof will involve several rather technical arguments, we have preferred to start giving all the geometrical details in the special case that $k=3$. We believe that this should help the reader to follow the various steps which will be necessary to handle the general case.

\textbf{Case $k =3$.} 

Using \eqref{tension-2s+1}, we write
\begin{equation}\label{tau3}
\tau_3(\varphi)=\bar\Delta^2\tau(\varphi)-\tr R^N(\bar{\Delta}\tau(\varphi),d\varphi(\cdot))d\varphi(\cdot)-\tr R^N(\bar{\nabla}_{(\cdot )}\tau(\varphi),\tau(\varphi))d\varphi(\cdot) \, .
\end{equation}
We fix the notation as follows:
\begin{eqnarray*}
\tau(\varphi)&=&u_0^\alpha \frac{\partial}{\partial y^\alpha} \\
\bar{\Delta}\tau(\varphi)&=& \left (\Delta u_0^\alpha+A_1^\alpha \right )\frac{\partial}{\partial y^\alpha}=  u_1^\alpha\frac{\partial}{\partial y^\alpha}\, ,
\end{eqnarray*}
with
\begin{equation}\label{def-A}
A_1^\alpha=A^\alpha \left ( u_0^\vartheta, \frac{\partial u_0^\vartheta}{\partial x^i}\right )\, ,
\end{equation}
where

$$
A=\left(A^{\alpha}\right):\R^n \times(\R^m\times \R^n)\to \R^n
$$
is defined, for $(\eta^{\vartheta},\xi^{\vartheta}_i)\in \R^n \times(\R^m\times \R^n)$, according to \eqref{eq:laplacian-on-sections}, by

\begin{eqnarray*}
A^\alpha \left (\eta^{\vartheta},\xi^{\vartheta}_i\right ) &=&\;\; \xi^{\vartheta}_i\, \left[-2\, g^{ij}\, \varphi_j^\beta\, \Gamma_{\beta \vartheta}^\alpha\right] \\\nonumber
&&+\eta^{\vartheta}  \left [ (\Delta \varphi^\beta)\, \Gamma_{\beta \vartheta}^\alpha - g^{ij}\, \varphi_j^\beta\, \varphi_i^\omega \; S^{\alpha}_{\beta \omega \vartheta} \right]\, .
\end{eqnarray*}

It is important to point out that $A^\alpha$ is linear with respect to $\eta^{\vartheta}$ and $\xi^{\vartheta}_i$. 

All of this can be iterated once more and yields
\[
\bar{\Delta}^2\tau(\varphi)= \left (\Delta u_1^\alpha+A_2^\alpha \right )\frac{\partial}{\partial y^\alpha}\, ,
\]
where now, of course,
\[
A_2^\alpha=A^\alpha \left ( u_1^\vartheta, \frac{\partial u_1^\vartheta}{\partial x^i}\right )\, .
\]
Now, let us assume that $\tau_3(\varphi)=0$. Using \eqref{tau3} and computing we find that this is locally equivalent to the following system of equations:
\begin{equation}\label{tau3+explicit}
\Delta u_1^\alpha = \left (F^3 \right )^\alpha \, ,
\end{equation}
where
\begin{eqnarray}\label{F3}\nonumber
\left(F^3 \right)^\alpha &=&-A^\alpha_2 -\Big [- \tr R^N(\bar{\Delta}\tau(\varphi),d\varphi(\cdot))d\varphi(\cdot)-\tr R^N(\bar{\nabla}_{(\cdot )}\tau(\varphi),\tau(\varphi))d\varphi(\cdot) \Big ]^\alpha\\
&=&-A^\alpha_2-g^{ij}\, u_1^\vartheta\, \varphi_i^\beta\, \varphi_j^\omega\, R^{\alpha}_{\omega \beta \vartheta}+g^{ij}\, \frac{\partial u_0^\beta}{\partial x^i}\, u_0^\vartheta\, \varphi_j^\omega \, R^{\alpha}_{\omega \beta \vartheta}\\ \nonumber
&&+g^{ij}\, u_0^\sigma\, u_0^\vartheta\, \varphi_i^\gamma\, \varphi_j^\omega\, \Gamma_{\gamma \sigma}^\beta\, R^{\alpha}_{\omega \beta \vartheta}\, .
\end{eqnarray}
Now, we define the $\R^n$-valued $1$-form
\[
v_0:= du_0=\left (du_0^1,\ldots,du_0^n \right )=\frac{\partial u_0^\alpha}{\partial x^i}\, dx^i \otimes e_\alpha \, ,
\]
where $\{e_\alpha\}$ is the canonical basis of $\R^n$ and $u_0$ is thought of as an $\R^n$-valued function defined on $U$. The components of the $1$-form $v_0$ give rise to an $\R^{mn}$-valued function defined on $U$:
\[
v_0= \left ( v_{0i}^\alpha \right )=\left ( \frac{\partial u_0^\alpha}{\partial x^i} \right ) \, , \qquad 1 \leq i \leq m,\quad 1 \leq \alpha \leq n \,.
\]
For simplicity, we keep the same notation for the $1$-form and the $\R^{mn}$-valued function. Note that here the index $i$ in $v_{0i}^\alpha$ does \textit{not} mean the derivative with respect to the variable $x^i$ ($v_{0}^\alpha$ does not even exist!). With this notation, \eqref{F3} can be rewritten as follows:
\begin{eqnarray}\label{F3-with-v0}
\left (F^3 \right )^\alpha &=&-A^\alpha_2-g^{ij}\, u_1^\vartheta\, \varphi_i^\beta\, \varphi_j^\omega\, R^{\alpha}_{\omega \beta \vartheta}+g^{ij}\, v_{0i}^\beta\, u_0^\vartheta\, \varphi_j^\omega\, R^{\alpha}_{\omega \beta \vartheta}\\ \nonumber
&&+g^{ij}\, u_0^\sigma\, u_0^\vartheta\, \varphi_i^\gamma\, \varphi_j^\omega\, \Gamma_{\gamma \sigma}^\beta\, R^{\alpha}_{\omega \beta \vartheta}\,.
\end{eqnarray}
Then, taking into account the definition of $A^\alpha_2$, we conclude that $F^3=\left(\left (F^3 \right )^\alpha\right)$ depends on $u_0,v_0,u_1$ and $\{\partial u_1^\alpha / \partial x^i\}_{i,\alpha}$. Moreover, for future use, we deduce from \eqref{def-A} and inspection of \eqref{F3-with-v0} that there exists $C>0$ such that on $D$ we have
\begin{equation}\label{F3-aro-estimate}
\left | F^3  \right | \leq C \left [\sum_{\vartheta}\left ( |u_0^\vartheta | +|u_1^\vartheta |\right )+\sum_{i,\vartheta} |v_{0i}^\vartheta |+\sum_{i,\vartheta}
\left |\frac{\partial u_{1}^\vartheta}{\partial x^i}  \right | \right ],
\end{equation}
where $D$ is an open subset of $M$ such that its closure is compact and contained in $U$.

Indeed, in \eqref{def-A} and \eqref{F3-with-v0} it is possible to bound from above by means of a constant any of the functions which appear as a multiplicative coefficient of $u_0,v_0,u_1$ and $\{\partial u_1^\alpha / \partial x^i\}_{i,\alpha}$. From this \eqref{F3-aro-estimate} can be obtained easily.

We shall also need to estimate the Laplacian of the $\R^{mn}$-valued function $v_0$. To this purpose, we perform a computation which gives the following output:
\begin{eqnarray}\label{Deltav0}
\nonumber
\left (\Delta v_0\right )_{i}^{\alpha} &=& \Delta \left ( \frac{\partial u_0^\alpha}{\partial x^i}\right ) \\
\nonumber &=& \frac{\partial\left ( \Delta u_0^\alpha \right )}{\partial x^i}+\frac{\partial g^{kj}}{\partial x^i}\frac{\partial^2 u_0^\alpha}{\partial x^k \partial x^j}
-\frac{\partial g^{\ell j}}{\partial x^i}\, \Gamma_{\ell j}^k\, \frac{\partial u_0^\alpha}{\partial x^k}-g^{\ell j}\, \frac{\partial \Gamma_{\ell j}^k}{\partial x^i}\frac{\partial u_0^\alpha}{\partial x^k}\\\nonumber
&=&\frac{\partial u_1^\alpha}{\partial x^i}+ \frac{\partial}{\partial x^i} 
\left\{ 2g^{\ell j}\, v_{0j}^\vartheta\, \varphi^\beta_{\ell}\, \Gamma^\alpha_{\beta \vartheta} 
-u_0^\vartheta  \left [ (\Delta \varphi^\beta)\, \Gamma_{\beta \vartheta}^\alpha -g^{\ell j}\, \varphi_j^\beta\, \varphi_{\ell}^\omega\; S^{\alpha}_{\beta \omega \vartheta}\right] \right \}\\
&&+\frac{\partial g^{kj}}{\partial x^i} \frac{\partial v_{0k}^\alpha}{\partial x^j}-\frac{\partial g^{\ell j}}{\partial x^i}\, \Gamma_{\ell j}^k\, v_{0k}^\alpha - g^{\ell j}\, \frac{\partial \Gamma_{\ell j}^k}{\partial x^i}\, v_{0k}^\alpha\, .
\end{eqnarray}
Now, we call $F^2=\left ( \left (F^2 \right )_i^\alpha\right )$ the right hand side of \eqref{Deltav0}. Thus, $F^2$ depends on $u_0,v_0,\{\partial v_{0i}^\alpha / \partial x^j\}_{i,j,\alpha}$ and $\{\partial u_1^\alpha / \partial x^i\}_{i,\alpha}$ and it is linear in each of them. Then, similarly to \eqref{F3-aro-estimate}, it is easy to deduce from \eqref{Deltav0} the following estimate on $D$:
\begin{equation}\label{F2-aro-estimate}
\left | F^2  \right | \leq C \left [\sum_{\vartheta} |u_0^\vartheta | + \sum_{i,\vartheta}|v_{0i}^\vartheta| + \sum_{i,\vartheta}
\left |\frac{\partial u_{1}^\vartheta}{\partial x^i}  \right | + \sum_{i,j,\vartheta}
\left |\frac{\partial v_{0i}^\vartheta}{\partial x^j}  \right |\right ]\, 
\end{equation}
for a suitably large constant $C>0$.

Finally, we also define the function $F^1=\left (\left (F^1 \right )^\alpha\right )$ as follows:
\[
\left (F^1\right )^\alpha =u_1^\alpha-A_1^\alpha\, .
\]
We note that $F^1$ depends on $u_0,v_0$ and $u_1$ and it is linear in each of them. It follows that there exists $C>0$ such that
\begin{equation}\label{F1-aro-estimate}
\left | F^1  \right | \leq C \left [\sum_{\vartheta}\left ( |u_0^\vartheta | + |u_1^\vartheta |\right ) + \sum_{i,\vartheta}|v_{0i}^\vartheta | \right ]\, .
\end{equation}
The next step is to consider $u_0,v_0,\{\partial v_{0i}^\alpha / \partial x^j\}_{i,j,\alpha},u_1$ and $\{\partial u_1^\alpha / \partial x^i\}_{i,\alpha}$ not as vector-valued functions defined on $U$, but as a set of independent variables. More specifically, we define
\begin{equation}\label{F-casok=3}
u=\left (u_0,v_0,u_1 \right )\in \R^{r_2} \quad {\rm and}\quad F=\left ( F^1,F^2,F^3\right ) \colon \mathbb{R}^{r_1}\to \mathbb{R}^{r_2}\, ,
\end{equation}
where $r_1=n+nm + n m^2 + n + nm$ and $r_2=n + n m + n$.
Thus, formally
\[
F=F \left ( u, \left (\frac{\partial u^a}{\partial x^i} \right )\right )\, .
\]
Next, if we think of $u$ as an $\mathbb{R}^{n(m+2)}$-valued function defined on $U$,  by construction we have $\Delta u = F$. Now, since by assumption $\varphi$ is harmonic on an open subset of $D$, it is clear that $u$ vanishes on that open subset. Moreover, since we have proved the estimates \eqref{F3-aro-estimate}, \eqref{F2-aro-estimate} and \eqref{F1-aro-estimate} on $D$, we can apply Theorem \ref{aro-theorem}. Then the statement follows precisely by the same globalization argument which was detailed in \cite[Proof of Theorem~1.3]{MR3990379}. Thus the proof of Theorem \ref{Th-harmonic-everywhere} is complete in the special case $k=3$.

\begin{Bem} 
\label{remark-differential}
In the previous proof, we have worked with the partial derivatives $\{\partial v_{0i}^\alpha / \partial x^j\}_{i,j,\alpha}$. This choice is the most suitable for the purposes of this paper and, particularly, for the extensions to the case $k \geq 4$ which will be illustrated below. However, we point out that it is also possible to work with covariant derivatives
\[
\left \{\nabla_j v_{0i}^\alpha\right \}_{i,j,\alpha}= \left \{ \frac{\partial v_{0i}^\alpha}{\partial x^j}-\Gamma_{ji}^k\, v_{0k}^\alpha \right \}_{i,j,\alpha}\, .
\]

Indeed, the difference
\[
\nabla_j v_{0i}^\alpha- \frac{\partial v_{0i}^\alpha}{\partial x^j}
\]
is linear in $\{ v_{0i}^\alpha\}_{i,\alpha}$. We also mention that $\nabla_iu^{\alpha}_1=\partial u^{\alpha}_1/\partial x^i$. Consequently, the basic estimates \eqref{F3-aro-estimate}, \eqref{F2-aro-estimate} and \eqref{F1-aro-estimate} still hold, conceptually with the same proofs. In particular, the choice of working with covariant derivatives was adopted in \cite{MR3990379}.
\end{Bem}

\begin{Bem} We point out that in \cite{MR3990379} the definition of $F^2=\left (\left (F^2 \right )_i^\alpha\right )$ should have been given as
\begin{equation}\label{eq-remark-cezar}
F^2= dw - \left (\frac{\partial g^{kj}}{\partial x^i}\frac{\partial v_{0k}^{\alpha}}{\partial x^j}-\left (\frac{\partial g^{\ell j}}{\partial x^i}\,  \Gamma_{\ell j}^k+g^{\ell j}\, \frac{\partial \Gamma_{\ell j}^k}{\partial x^i} \right ) v_{0k}^\alpha \right )\, .
\end{equation}
Note that a different sign convention was used in \cite{MR3990379}.
Since the additional term is linear in $\{ v_{0i}^\alpha\}_{i,\alpha}$ and $\{\partial v_{0i}^\alpha / \partial x^j\}_{i,j,\alpha}$, or $\{\nabla_j v_{0i}^\alpha\}_{i,j,\alpha}$, the proofs in \cite{MR3990379} can be completed with minor changes.
\end{Bem}
\textbf{Case $k \geq 4$.}
First, we provide a short illustration of the case $k=4$ because this step may help to understand the idea behind the introduction of a suitable set of recursively defined new variables, a fact which is a key point.

We recall from \eqref{tension-2s} that, when $k=4$, the $4$-tension field is 
\begin{eqnarray}\label{4-tension-explicit}
\tau_4(\varphi)&=&\bar\Delta^3\tau(\varphi)-\tr R^N(\bar{\Delta}^2\tau(\varphi),d\varphi(\cdot))d\varphi(\cdot)\\\nonumber
&&-\tr R^N(\bar{\nabla}_{(\cdot)}\bar{\Delta}\tau(\varphi),\tau(\varphi))d\varphi(\cdot)+\tr R^N(\bar{\Delta}\tau(\varphi),\bar{\nabla}_{(\cdot)}\tau(\varphi))d\varphi(\cdot) \, .
\end{eqnarray}
Now, let us assume that $\tau_4(\varphi)=0$. First, we set again
\begin{eqnarray*}
\tau(\varphi)&=&u_0^\alpha \frac{\partial}{\partial y^\alpha} \\
\bar{\Delta}\tau(\varphi)&=& \left (\Delta u_0^\alpha+A_1^\alpha \right )\frac{\partial}{\partial y^\alpha}=  u_1^\alpha\frac{\partial}{\partial y^\alpha}\, ,
\end{eqnarray*}
where
\[
A_1^\alpha=A^\alpha \left ( u_0^\vartheta, \frac{\partial u_0^\vartheta}{\partial x^i}\right )
\]
is defined in \eqref{def-A}. Analogously, we continue with
\begin{eqnarray*}
\bar{\Delta}^2\tau(\varphi)&=&\left (\Delta u_1^\alpha+A_2^\alpha \right )\frac{\partial}{\partial y^\alpha}=  u_2^\alpha\frac{\partial}{\partial y^\alpha} \\
\bar{\Delta}^3\tau(\varphi)&=& \left (\Delta u_2^\alpha+A_3^\alpha \right )\frac{\partial}{\partial y^\alpha}\, ,
\end{eqnarray*}
where \begin{eqnarray*}
A_j^\alpha&=&A^\alpha \left ( u_{j-1}^\vartheta, \frac{\partial u_{j-1}^\vartheta}{\partial x^i}\right )\, , \quad j=1,2,3 \,.
\end{eqnarray*}
Now, using \eqref{4-tension-explicit}, we find that in our notation the assumption $\tau_4(\varphi)=0$ is equivalent to
\[
\left (\Delta u_2^\alpha \right )\frac{\partial}{\partial y^\alpha}= \left (F^4 \right )^\alpha \frac{\partial}{\partial y^\alpha}\,,
\]
where we have set:
\begin{eqnarray}\label{def-F4}
\left(F^4 \right )^\alpha&=&-A_3^\alpha- \Big [-\tr R^N(\bar{\Delta}^2\tau(\varphi),d\varphi(\cdot))d\varphi(\cdot) \\\nonumber
&&-\tr R^N(\bar{\nabla}_{(\cdot)}\bar{\Delta}\tau(\varphi),\tau(\varphi))d\varphi(\cdot)+\tr R^N(\bar{\Delta}\tau(\varphi),\bar{\nabla}_{(\cdot)}\tau(\varphi))d\varphi(\cdot) \Big ]^\alpha \,.
\end{eqnarray}
We use the following set of variables:
\begin{eqnarray*}
u_0&=& \left (u_0 ^\alpha\right )\,, \quad v_0:=du_0=\nabla u_0 \,;\\
u_1&=& \left (u_1 ^\alpha\right )\,, \quad v_1:=\nabla u_1 \,;\\
u_2&=& \left (u_2 ^\alpha\right ) \,.
\end{eqnarray*}
We define
\[
u=\left ( \begin{array}{l}
u_0\\
v_0\\
u_1\\
v_1\\
u_2
\end{array} \right )
\]
and $F=F\left (u_0,v_0,\nabla v_0, u_1,v_1,\nabla v_1,u_2,\nabla u_2\right )$ as follows:
\[
F=\left ( \begin{array}{l}
u_1-A_1\\
d\left (u_1-A_1\right )\\
u_2-A_2\\
d\left (u_2-A_2\right )\\
F^4
\end{array} \right )
\]
Now, $\Delta u = F\, +$ terms linear in $v_0$ and its first derivatives for the second component of $F$ and $+$  terms linear in $v_1$ and its first derivatives for the fourth component as in \eqref{eq-remark-cezar}. Using the same technique that we employed for the first $2$ vector components of $F$ in the case $k=3$, now the first $2$ vector components of $F$ and, analogously, the vector components $3,4$, can be estimated to ensure the validity of \eqref{aro-voraus}.
As for $F^4$, the explicit analysis of \eqref{def-F4} yields
\begin{equation}\label{def-F4-coor-loc}
\begin{split}
\left (F^4 \right )^\alpha=&-A_3^\alpha-g^{ij}u_2^\vartheta \varphi_i^\beta \varphi_j^\omega\,  R^\alpha_{\omega \beta \vartheta}\\
&+g^{ij}\,v_{1i}^\beta\, u_0^\vartheta \varphi_j^\omega \, R^\alpha_{\omega \beta \vartheta}\\
&+g^{ij}u_1^\sigma u_0^\vartheta \varphi_i^\gamma\varphi_j^\omega\,  \Gamma^\beta_{\gamma \sigma}  \,R^\alpha_{\omega \beta \vartheta}\\
&+g^{ij}\,v_{0i}^\beta \,u_1^\vartheta \varphi_j^\omega \, R^\alpha_{\omega \beta \vartheta}\\
&+g^{ij}u_0^\sigma u_1^\vartheta \varphi_i^\gamma\, \varphi_j^\omega\,  \Gamma^\beta_{\gamma \sigma} \, R^\alpha_{\omega \beta \vartheta}\end{split}
\end{equation}
from which \eqref{aro-voraus} follows easily. This ends the case $k=4$. The general case can be handled similarly. Indeed, for any fixed value $k \geq4$, we have recursively defined functions $u_i=\left ( u_i^\alpha\right )$ by means of
\begin{equation}\label{Delta-i-tau=u-i}
\bar{\Delta}^{i+1}\tau(\varphi)= \left (\Delta u_{i}^\alpha+A_{i+1}^\alpha \right )\frac{\partial}{\partial y^\alpha}=  u_{i+1}^\alpha\frac{\partial}{\partial y^\alpha}\,, 
\end{equation}
for $0 \leq i \leq k-2$ and also $A_{k-1}^\alpha$ is defined. Then we introduce the vector-valued function
\begin{equation}\label{def-u-general-th-1-1}
u=\left ( \begin{array}{l}
u_0\\
v_0\\
u_1\\
v_1\\
\vdots \\
v_{k-3}\\
u_{k-2}
\end{array} \right )
\end{equation}
where $v_i= \nabla u_i$, $0 \leq i \leq k-3$.  Note that we do not introduce $v_{k-2}$ in \eqref{def-u-general-th-1-1}. Then we define
$$F=F\left (u_0,v_0,\nabla v_0,\ldots, u_{k-3},v_{k-3},\nabla v_{k-3},u_{k-2},\nabla u_{k-2}\right )$$
as follows:
\[
F=\left ( \begin{array}{l}
u_1-A_1\\
d\left (u_1-A_1\right )\\
u_2-A_2\\
d\left (u_2-A_2\right )\\
\vdots\\
d\left (u_{k-2}-A_{k-2}\right )\\
F^k
\end{array} \right ) \,,
\]
with, similarly to \eqref{def-F4}, $\left (F^k\right )^\alpha=-A_{k-1}^\alpha-[\,\ldots\,]^\alpha$. More precisely, here $[\,\ldots\,]$ is the right-hand side of \eqref{tension-2s} or \eqref{tension-2s+1} without the first term. Again, by construction 
$$
\Delta u =F+\textrm{~terms linear in~} v_j,\nabla v_j,\quad  j=0,\ldots k-3
$$ 
and it is easy to see that the first $2(k-2)$ vector components of $F$ can be estimated in such a way that \eqref{aro-voraus} holds. Finally, direct inspection of \eqref{tension-2s} or \eqref{tension-2s+1} (compare with \eqref{def-F4-coor-loc}) shows that also $F^k$ can be estimated so that \eqref{aro-voraus} is verified. So the proof ends by application of Theorem \ref{aro-theorem}. 
\end{proof}

\begin{proof}[Proof of Theorem \ref{main-theorem}]

Let $\varphi, \tilde{\varphi}$ be two $k$-harmonic maps which coincide on an open subset. To simplify the notation, we shall also denote by $\varphi, \tilde{\varphi}$ the vector-valued functions which represent $\varphi, \tilde{\varphi}$ in local charts:
\[
\varphi= \left(\varphi^1,\ldots,\varphi^n \right ) \,; \quad \tilde \varphi= \left(\tilde\varphi^1,\ldots,\tilde\varphi^n \right ) \,.
\]

We define a vector-valued function $u$ for the map \(\varphi\) as follows
\begin{equation}\label{def-u-general}
u=\left ( \begin{array}{l}
\varphi\\
d \varphi \\
u_0\\
v_0\\
u_1\\
v_1\\
\vdots \\
v_{k-3}\\
u_{k-2}
\end{array} \right )\,,
\end{equation}
where $u_0,\ldots,u_{k-2}$ and $v_0,\ldots,v_{k-3}$ are defined as in the proof of Theorem~\ref{Th-harmonic-everywhere}.
We also need to introduce the analogous vector-valued function associated with $\tilde{\varphi}$, i.e.,
\begin{equation}\label{def-u-tilda-general}
\tilde u=\left ( \begin{array}{l}
\tilde \varphi\\
d \tilde\varphi \\
\tilde u_0\\
\tilde v_0\\
\tilde u_1\\
\tilde v_1\\
\vdots \\
\tilde v_{k-3}\\
\tilde u_{k-2}
\end{array} \right )\,.
\end{equation}
Note that, in contrast to the proof of Theorem~\ref{Th-harmonic-everywhere}, the functions \eqref{def-u-general}
and \eqref{def-u-tilda-general} also contain \(\varphi\), $\tilde{\varphi}$ and their first derivatives.

The proof of Theorem \ref{main-theorem} amounts to showing that we can apply Aronszajn's Theorem,
that is Theorem \ref{aro-theorem},
to the vector-valued function 
$$z=u-\tilde u\,,$$
with $u, \,\tilde u$ defined in \eqref{def-u-general} and \eqref{def-u-tilda-general} respectively. 
As in the proof of Theorem 1.1, the following functions are also defined:
\[
\begin{array}{l}
A_i= \left (A_i ^\alpha\right )\,, \quad \tilde A_i= \left (\tilde A_i ^\alpha\right )\,,\quad (1 \leq i \leq k-2) \\
\left (F^k\right )^\alpha=-A_{k-1}^\alpha-[\,\ldots\,]^\alpha\,, \quad
\left (\tilde F^k\right )^\alpha=-\tilde A_{k-1}^\alpha-[\,\tilde{\ldots}\,]^\alpha\,. 
\end{array} 
\] 
Next, we define $G$ as follows:
\begin{equation}\label{def-G}
G=\left ( \begin{array}{l}
\Delta \left ( \varphi- \tilde \varphi \right )\\
\Delta \left ( d\varphi- d\tilde \varphi \right )\\
\left (u_1-A_1\right )-\left (\tilde u_1- \tilde A_1\right )\\
d\left (u_1-A_1\right )- d\left (\tilde u_1- \tilde A_1\right )\\
\left (u_2-A_2\right )-\left (\tilde u_2- \tilde A_2\right )\\
d\left (u_2-A_2\right )- d\left (\tilde u_2- \tilde A_2\right )\\
\vdots\\
d\left (u_{k-2}-A_{k-2}\right )-d\left (\tilde u_{k-2}- \tilde A_{k-2}\right )\\
F^k -\tilde F ^k
\end{array} \right ) 
\end{equation}
Now, by construction, we have
\[
\Delta z = G+\textrm{~terms linear in~} (v_j-\tilde v_j) ,(\nabla v_j-\nabla \tilde v_j), \quad j=0,\ldots k-3.
\]

Note that there are three different blocks in the definition of \(G\).
The first two rows only contain the Laplacian applied to \(\varphi\) and its first partial derivatives,
and we will explain in more detail below how one should think of \(d\varphi\) in the the definition of \(G\).
After that, we always have pairs of \(u_j-A_j\) and its first derivatives, with \(1\leq j\leq k-2\).
In the last row we have the right hand side of the polyharmonic map equation \eqref{tension-2s}, \eqref{tension-2s+1},
denoted by \(F^k\).
Note that this is the only place in \eqref{def-G} where the Euler-Lagrange equation for polyharmonic maps
enters.

Our aim is now to apply the theorem of Aronszajn, that is Theorem \ref{aro-theorem}, to \(z\).
Hence, we have to estimate \(\Delta z\) (equivalently, $G$) in terms of \(z\) and its first partial derivatives.

In the following \(C\) will always represent a positive constant whose value may change from line to line.

To estimate the first row in \eqref{def-G} we use the following 
\begin{Lem}
Let \(\varphi,\tilde\varphi\) be two maps with corresponding variables \eqref{def-u-general} and \eqref{def-u-tilda-general}.
Then the following estimate holds
\begin{align}
\label{estimate-Delta-varphi}
|\Delta (\varphi-\tilde\varphi)|\leq C(|\varphi-\tilde\varphi|+|d\varphi-d\tilde\varphi|+|u_0-\tilde u_0|).
\end{align}
\end{Lem}
\begin{proof}
Recall that 
\begin{align*}
\Delta\varphi^\alpha=-u_0^\alpha+\langle d\varphi^\beta,d\varphi^\gamma\rangle\Gamma^\alpha_{\beta\gamma}\,,
\end{align*}
where, here and below, we have shortened the notation denoting
\[
\langle d\varphi^\beta,d\varphi^\gamma\rangle=g^{ij}\, \frac{\partial\varphi^\beta}{\partial x^i} \frac{\partial\varphi^\gamma}{\partial x^j}\,.
\]
We can rewrite 
\begin{align*}
\Delta(\varphi^\alpha-\tilde\varphi^\alpha)=&-u_0^\alpha+\tilde u_0^\alpha+\langle d\varphi^\beta,d\varphi^\gamma\rangle\Gamma^\alpha_{\beta\gamma}(\varphi)
-\langle d\tilde\varphi^\beta,d\tilde\varphi^\gamma\rangle\Gamma^\alpha_{\beta\gamma}(\tilde\varphi) \\
=&-u_0^\alpha+\tilde u_0^\alpha \\
&+\langle d\varphi^\beta-d\tilde\varphi^\beta,d\varphi^\gamma\rangle\Gamma^\alpha_{\beta\gamma}(\varphi)
+\langle d\tilde\varphi^\beta,d\varphi^\gamma-d\tilde\varphi^\gamma\rangle\Gamma^\alpha_{\beta\gamma}(\varphi) \\
&+\langle d\tilde\varphi^\beta,d\tilde\varphi^\gamma\rangle\big(\Gamma^\alpha_{\beta\gamma}(\varphi)-\Gamma^\alpha_{\beta\gamma}(\tilde\varphi)\big).
\end{align*}
The first three terms on the right hand side can be estimated directly.
To estimate the difference of the Christoffel symbols we make use of the mean-value inequality,
for more details we refer to the discussion before Lemma 2.6 in \cite{MR3990379}.
The proof is now complete.
\end{proof}

In the following we will often apply the mean-value inequality without explicitly mentioning it.

As a second step, we estimate the second line of \eqref{def-G}.
\begin{Lem}
\label{lemma-estimate-differential}
Let \(\varphi,\tilde\varphi\) be two maps with corresponding variables \eqref{def-u-general} and \eqref{def-u-tilda-general}.
Then the following estimate holds
\begin{align}
\label{estimate-differential}
|\Delta (d\varphi-\tilde d\varphi)|\leq C\big(|\varphi-\tilde\varphi|+|d\varphi-d\tilde\varphi|+|\nabla d\varphi-\nabla d\tilde\varphi|+|v_0-\tilde v_0|\big).
\end{align}
\end{Lem}
\begin{proof}
Recall that we use \(d\varphi,d\tilde\varphi\) in \eqref{def-G} to represent the partial derivatives of \(\varphi\) and \(\tilde\varphi\).
Hence, when we apply the Laplacian to the second line in \eqref{def-G} we will get some correction terms as 
already computed in \eqref{Deltav0}, see also Remark \ref{remark-differential}.

Now, let us consider our two maps \(\varphi,\tilde\varphi\): by combining \eqref{Deltav0} and using that \(\Delta\varphi^\alpha=-u_0^\alpha+\langle d\varphi^\beta,d\varphi^\gamma\rangle\Gamma^\alpha_{\beta\gamma}\) we find
\begin{align*}
\Delta(\varphi^\alpha_i-\tilde\varphi^\alpha_i)=&
-v^\alpha_{0i}+\tilde v^\alpha_{0i}
+2\langle\nabla_i d\varphi^\beta,d\varphi^\gamma\rangle\Gamma^\alpha_{\beta\gamma}(\varphi)
-2\langle\nabla_i d\tilde\varphi^\beta,d\tilde\varphi^\gamma\rangle\Gamma^\alpha_{\beta\gamma}(\tilde\varphi)\\
&+\langle d\varphi^\beta,d\varphi^\gamma\rangle\frac{\partial\Gamma^\alpha_{\beta\gamma}(\varphi)}{\partial y^\delta}\varphi^\delta_i
-\langle d\tilde\varphi^\beta,d\tilde\varphi^\gamma\rangle\frac{\partial\Gamma^\alpha_{\beta\gamma}(\tilde\varphi)}{\partial y^\delta}
\tilde\varphi^\delta_i\\
&+\frac{\partial g^{k\ell}}{\partial x^i}\big(\frac{\partial\varphi_k^\alpha}{\partial x^\ell}-\frac{\partial\tilde\varphi_k^\alpha}{\partial x^\ell}\big)
-\big(\frac{\partial g^{k\ell}}{\partial x^i}\Gamma^j_{k\ell}
+g^{k\ell}\frac{\partial\Gamma^j_{k\ell}}{\partial x^i}\big)\big(\varphi^\alpha_j-\tilde\varphi^\alpha_j\big)\,,
\end{align*}
where in the last line, for clarity, we have also added the linear terms.
In order to estimate the second term on the right hand side we rewrite
\begin{align*}
\langle\nabla_i d\varphi^\beta,d\varphi^\gamma\rangle\Gamma^\alpha_{\beta\gamma}(\varphi)&
-\langle\nabla_i d\tilde\varphi^\beta,d\tilde\varphi^\gamma\rangle\Gamma^\alpha_{\beta\gamma}(\tilde\varphi) \\
=&\langle\nabla_i d\varphi^\beta-\nabla_i d\tilde\varphi^\beta,d\varphi^\gamma\rangle\Gamma^\alpha_{\beta\gamma}(\varphi)
+\langle\nabla_i d\tilde\varphi^\beta,d\varphi^\gamma-d\tilde\varphi^\gamma\rangle\Gamma^\alpha_{\beta\gamma}(\varphi) \\
&+\langle\nabla_i d\tilde\varphi^\beta,d\tilde\varphi^\gamma\rangle
\big(\Gamma^\alpha_{\beta\gamma}(\varphi)-\Gamma^\alpha_{\beta\gamma}(\tilde\varphi)\big).
\end{align*}
Then, it is easy to estimate
\begin{align*}
|\langle\nabla_i d\varphi^\beta,d\varphi^\gamma\rangle\Gamma^\alpha_{\beta\gamma}(\varphi)
-\langle\nabla_i d\tilde\varphi^\beta,d\tilde\varphi^\gamma\rangle\Gamma^\alpha_{\beta\gamma}(\tilde\varphi)|
\leq C(|\varphi-\tilde\varphi|+|d\varphi-d\tilde\varphi|+|\nabla d\varphi-\nabla d\tilde\varphi|).
\end{align*}
Again, we rewrite
\begin{align*}
\langle d\varphi^\beta,d\varphi^\gamma\rangle&\frac{\partial\Gamma^\alpha_{\beta\gamma}(\varphi)}{\partial y^\delta}\varphi^\delta_i
-\langle d\tilde\varphi^\beta,d\tilde\varphi^\gamma\rangle\frac{\partial\Gamma^\alpha_{\beta\gamma}(\tilde\varphi)}{\partial y^\delta}\tilde\varphi^\delta_i \\
=&\langle d\varphi^\beta-d\tilde\varphi^\beta,d\varphi^\gamma\rangle\frac{\partial\Gamma^\alpha_{\beta\gamma}(\varphi)}{\partial y^\delta}\varphi^\delta_i
+\langle d\tilde\varphi^\beta,d\varphi^\gamma-d\tilde\varphi^\gamma\rangle\frac{\partial\Gamma^\alpha_{\beta\gamma}(\varphi)}{\partial y^\delta}\varphi^\delta_i\\
&+\langle d\tilde\varphi^\beta,d\tilde\varphi^\gamma\rangle
\big(\frac{\partial\Gamma^\alpha_{\beta\gamma}(\varphi)}{\partial y^\delta}-\frac{\partial\Gamma^\alpha_{\beta\gamma}(\tilde\varphi)}{\partial y^\delta}\big)\varphi^\delta_i
+\langle d\tilde\varphi^\beta,d\tilde\varphi^\gamma\rangle\frac{\partial\Gamma^\alpha_{\beta\gamma}(\tilde\varphi)}{\partial y^\delta}
(\varphi^\delta_i-\tilde\varphi^\delta_i).
\end{align*}
We deduce the estimate
\begin{align*}
|\langle d\varphi^\beta,d\varphi^\gamma\rangle\frac{\partial\Gamma^\alpha_{\beta\gamma}(\varphi)}{\partial y^\delta}\varphi^\delta_i
-\langle d\tilde\varphi^\beta,d\tilde\varphi^\gamma\rangle\frac{\partial\Gamma^\alpha_{\beta\gamma}(\tilde\varphi)}{\partial y^\delta}\tilde\varphi^\delta_i|
\leq C(|\varphi-\tilde\varphi|+|d\varphi-d\tilde\varphi|).
\end{align*}
The claim now follows by combining the equations.
\end{proof}
In the following two lemmata we will estimate the pairs \(u_{j+1}-A_{j+1}\) and their derivatives which are in the middle block of \eqref{def-G}.

\begin{Lem}
Let \(\varphi,\tilde\varphi\) be two maps with corresponding variables \eqref{def-u-general} and \eqref{def-u-tilda-general}. Assume that \(0 \leq j\leq k-3\).
Then the following estimate holds
\begin{align}
\label{estimate-A}
\big|\big(u_{j+1}&-A_{j+1}\big)-\big(\tilde u_{j+1}-\tilde A_{j+1}\big)\big| \\
\nonumber&\leq C(|\varphi-\tilde\varphi|+|d\varphi-d\tilde\varphi|+|u_0-\tilde u_0|+|u_j-\tilde u_j|+|u_{j+1}-\tilde u_{j+1}|+|v_j-\tilde v_j|)\,.
\end{align}
\end{Lem}
\begin{proof}
We know that (this follows from \eqref{def-A} and using \(\Delta\varphi^\alpha=-u_0^\alpha+\langle d\varphi^\beta,d\varphi^\gamma\rangle\Gamma^\alpha_{\beta\gamma}\))
\begin{align*}
A_{j+1}^\alpha=&
-2\langle du_j^\nu,d\varphi^\beta\rangle\Gamma^\alpha_{\beta\nu}
-u_j^\nu u_0^\beta\Gamma_{\beta\nu}^\alpha
+u_j^\nu\langle d\varphi^\vartheta,d\varphi^\sigma\rangle C_{\vartheta\sigma\nu}^\alpha,
\end{align*}
where 

\begin{align*}
C_{\vartheta\sigma\nu}^\alpha:=
\Gamma^\mu_{\vartheta\sigma}\Gamma^\alpha_{\mu\nu}
-S^\alpha_{\vartheta\sigma\nu},
\end{align*}
with $S$ defined in \eqref{eq:definitionofS}.
Hence, for two maps \(\varphi,\tilde\varphi\) we get
\begin{align*}
A_{j+1}^\alpha-\tilde A_{j+1}^\alpha=&
-2\langle du_j^\nu,d\varphi^\beta\rangle\Gamma^\alpha_{\beta\nu}(\varphi)
+2\langle d\tilde u_j^\nu,d\tilde\varphi^\beta\rangle\Gamma^\alpha_{\beta\nu}(\tilde\varphi) \\
&-u_j^\nu u_0^\beta\Gamma_{\beta\nu}^\alpha(\varphi)
+\tilde u_j^\nu \tilde u_0^\beta\Gamma_{\beta\nu}^\alpha(\tilde\varphi)\\
&+u_j^\nu\langle d\varphi^\vartheta,d\varphi^\sigma\rangle C_{\vartheta\sigma\nu}^\alpha(\varphi)
-\tilde u_j^\nu\langle d\tilde\varphi^\vartheta,d\tilde\varphi^\sigma\rangle C_{\vartheta\sigma\nu}^\alpha(\tilde\varphi).
\end{align*}
Now, we rewrite
\begin{align*}
-\langle du_j^\nu&,d\varphi^\beta\rangle\Gamma^\alpha_{\beta\nu}(\varphi)
+\langle d\tilde u_j^\nu,d\tilde\varphi^\beta\rangle\Gamma^\alpha_{\beta\nu}(\tilde\varphi) \\
=&-\langle du_j^\nu-d\tilde u_j^\nu,d\varphi^\beta\rangle\Gamma^\alpha_{\beta\nu}(\varphi)
-\langle d\tilde u_j^\nu,d\varphi^\beta-d\tilde\varphi^\beta\rangle\Gamma^\alpha_{\beta\nu}(\varphi)
-\langle d\tilde u_j^\nu,d\tilde\varphi^\beta\rangle
\big(\Gamma^\alpha_{\beta\nu}(\varphi)-\Gamma^\alpha_{\beta\nu}(\tilde\varphi)\big).
\end{align*}
This gives the estimate
\begin{align*}
|\langle du_j^\nu,d\varphi^\beta\rangle\Gamma^\alpha_{\beta\nu}(\varphi)
-\langle d\tilde u_j^\nu,d\tilde\varphi^\beta\rangle\Gamma^\alpha_{\beta\nu}(\tilde\varphi)| 
\leq C(|\varphi-\tilde\varphi|+|d\varphi-d\tilde\varphi|+|du_j-d\tilde u_j|).
\end{align*}
Again, we rewrite
\begin{align*}
-u_j^\nu u_0^\beta\Gamma_{\beta\nu}^\alpha(\varphi)
+\tilde u_j^\nu \tilde u_0^\beta\Gamma_{\beta\nu}^\alpha(\tilde\varphi)
=&-(u_j^\nu-\tilde u_j^\nu) u_0^\beta\Gamma_{\beta\nu}^\alpha(\varphi)
-\tilde u_j^\nu (u_0^\beta-\tilde u_0^\beta)\Gamma_{\beta\nu}^\alpha(\varphi) \\
&-\tilde u_j^\nu\tilde u_0^\beta\big(\Gamma_{\beta\nu}^\alpha(\varphi)-\Gamma_{\beta\nu}^\alpha(\tilde\varphi)\big)
\end{align*}
and estimate
\begin{align*}
|u_j^\nu u_0^\beta\Gamma_{\beta\nu}^\alpha(\varphi)
-\tilde u_j^\nu \tilde u_0^\beta\Gamma_{\beta\nu}^\alpha(\tilde\varphi)|
\leq C(|\varphi-\tilde\varphi|+|u_0-\tilde u_0|+|u_j-\tilde u_j|).
\end{align*}
Finally, we rewrite
\begin{align*}
u_j^\nu\langle d\varphi^\vartheta,&d\varphi^\sigma\rangle C_{\vartheta\sigma\nu}^\alpha(\varphi)
-\tilde u_j^\nu\langle d\tilde\varphi^\vartheta,d\tilde\varphi^\sigma\rangle C_{\vartheta\sigma\nu}^\alpha(\tilde\varphi)\\
=&(u_j^\nu-\tilde u_j^\nu)\langle d\varphi^\vartheta,d\varphi^\sigma\rangle C_{\vartheta\sigma\nu}^\alpha(\varphi)
+\tilde u_j^\nu\langle d\varphi^\vartheta-d\tilde\varphi^\vartheta,d\varphi^\sigma\rangle C_{\vartheta\sigma\nu}^\alpha(\varphi) \\
&+\tilde u_j^\nu\langle d\tilde\varphi^\vartheta,d\varphi^\sigma-d\tilde\varphi^\sigma\rangle C_{\vartheta\sigma\nu}^\alpha(\varphi)
+\tilde u_j^\nu\langle d\tilde\varphi^\vartheta,d\tilde\varphi^\sigma\rangle \big(C_{\vartheta\sigma\nu}^\alpha(\varphi)- C_{\vartheta\sigma\nu}^\alpha(\tilde\varphi)\big).
\end{align*}
Hence, we find the estimate
\begin{align*}
|u_j^\nu\langle d\varphi^\vartheta,&d\varphi^\sigma\rangle C_{\vartheta\sigma\nu}^\alpha(\varphi)
-\tilde u_j^\nu\langle d\tilde\varphi^\vartheta,d\tilde\varphi^\sigma\rangle C_{\vartheta\sigma\nu}^\alpha(\tilde\varphi)|
\leq C(|\varphi-\tilde\varphi|+|d\varphi- d\tilde\varphi|+|u_j-\tilde u_j|).
\end{align*}
This completes the proof.
\end{proof}

\begin{Lem}
Let \(\varphi,\tilde\varphi\) be two maps with corresponding variables \eqref{def-u-general} and \eqref{def-u-tilda-general}. Assume that \(0 \leq j\leq k-3\).
Then the following estimate holds
\begin{align}
\label{estimate-dA}
\big|d\big(u_{j+1}&-A_{j+1}\big)-d\big(\tilde u_{j+1}-\tilde A_{j+1}\big)\big| \\
\nonumber\leq & 
C\big(|\varphi-\tilde\varphi|+|d\varphi-\tilde d\varphi|+|\nabla d\varphi-\nabla d\tilde\varphi| 
+|u_0-\tilde u_0|+|v_0-\tilde v_0|+|u_j-\tilde u_j| \\
\nonumber&+|v_j-\tilde v_j|+|\nabla v_j-\nabla\tilde v_j|+|u_{j+1}-\tilde u_{j+1}|\big)\,.
\end{align}
\end{Lem}
\begin{proof}
By a direct calculation we find
\begin{align*}
\nabla_i A^\alpha_{j+1}=&-2\langle\nabla_iv_j^\gamma,d\varphi^\beta\rangle\Gamma^\alpha_{\beta\gamma}
-2\langle v_j^\gamma,\nabla_i d\varphi^\beta\rangle\Gamma^\alpha_{\beta\gamma}
-2\langle v_j^\gamma,d\varphi^\beta\rangle\frac{\partial\Gamma^\alpha_{\beta\gamma}}{\partial y^\delta}\varphi^\delta_i \\
&-\nabla_iu_j^\gamma\, u_0^\beta\,\Gamma^\alpha_{\beta\gamma}-u_j^\gamma \,\nabla_iu_0^\beta\,\Gamma^\alpha_{\beta\gamma}
-u_j^\gamma\, u_0^\beta\,\frac{\partial\Gamma^\alpha_{\beta\gamma}}{\partial y^\delta}\varphi^\delta_i \\
&+\nabla_i u^\nu_j\langle d\varphi^\vartheta,d\varphi^\sigma\rangle C^\alpha_{\vartheta\sigma\nu}
+ u^\nu_j\langle\nabla_i d\varphi^\vartheta,d\varphi^\sigma\rangle C^\alpha_{\vartheta\sigma\nu}
+ u^\nu_j\langle d\varphi^\vartheta,\nabla_id\varphi^\sigma\rangle C^\alpha_{\vartheta\sigma\nu} \\
&+ u^\nu_j\langle d\varphi^\vartheta,d\varphi^\sigma\rangle \frac{\partial C^\alpha_{\vartheta\sigma\nu}}{\partial y^\delta}\varphi^\delta_i.
\end{align*}
Again, we have to be careful when applying the Laplacian to \(v_j=\nabla u_j\)
as the Laplacian does not commute with partial derivatives and we get several
extra terms as demonstrated in \eqref{Deltav0}.
However, all these terms on the right hand side can be easily estimated
in terms of \(v_r,\nabla v_r\).

The statement of the lemma can now be derived as in the previous lemmata.
\end{proof}

Finally, we estimate the contribution in \eqref{def-G} originating from the polyharmonic map equation.
We only consider the case of polyharmonic maps of even order with \(2s=k\geq 4\) as the odd case 
follows by exactly the same arguments.
\begin{Lem}

Let \(\varphi\colon M\to N\) be a polyharmonic map of even order with corresponding variables \eqref{def-u-general}.
Then
\begin{align}
\label{definition-Fk-even}
(F^k)^\alpha=&-A^\alpha_{k-1}
-u^\delta_{k-2}\langle d\varphi^\gamma,d\varphi^\beta\rangle R^\alpha_{\beta\gamma\delta} \\
\nonumber&+\sum_{\ell=1}^{\frac{k}{2}-1}\big(u^\delta_{\frac{k}{2}-\ell-1}\langle v^\gamma_{\frac{k}{2}+\ell-2},d\varphi^\beta\rangle R^\alpha_{\beta\gamma\delta}
+u^\vartheta_{\frac{k}{2}+\ell-2}u^\delta_{\frac{k}{2}-\ell-1}\langle d\varphi^\eta,d\varphi^\beta\rangle E^\alpha_{\beta\delta\vartheta\eta} \\
\nonumber &\hspace{1.5cm} +u^\delta_{\frac{k}{2}+\ell-2}\langle v^\gamma_{\frac{k}{2}-\ell-1},d\varphi^\beta\rangle R^\alpha_{\beta\gamma\delta}\big),
\end{align}
where \(E^\alpha_{\beta\delta\vartheta\eta}:=R^\alpha_{\beta\gamma\delta}\Gamma^\gamma_{\vartheta\eta}{+R^\alpha_{\beta\gamma\vartheta}\Gamma^\gamma_{\delta\eta}}\).
\end{Lem}
\begin{proof}
This follows directly from the Euler-Lagrange equation \eqref{tension-2s} using
\begin{align*}
\big(R^N(\bar\nabla_{e_j}\bar\Delta^{\frac{k}{2}+\ell-2}\tau(\varphi),\Delta^{\frac{k}{2}-\ell-1}\tau(\varphi))d\varphi(e_i)\big)^\alpha
=&u^\delta_{\frac{k}{2}-\ell-1}\langle v^\gamma_{\frac{k}{2}+\ell-2},d\varphi^\beta\rangle R^\alpha_{\beta\gamma\delta} \\
&+u^\vartheta_{\frac{k}{2}+\ell-2}u^\delta_{\frac{k}{2}-\ell-1}\langle d\varphi^\eta,d\varphi^\beta\rangle R^\alpha_{\beta\gamma\delta}\Gamma^\gamma_{\vartheta\eta}.
\end{align*}

\end{proof}

\begin{Lem}
Suppose \(\varphi,\tilde\varphi\) are two polyharmonic maps with corresponding variables \eqref{def-u-general} and \eqref{def-u-tilda-general}.
Then the following estimate holds
\begin{align}
\label{estimate-Fk}
|(F^k)^\alpha-(\tilde F^k)^\alpha|\leq
C\Big(|\varphi-\tilde\varphi|+|d\varphi-d\tilde\varphi|+\sum_{\ell=0}^{k-2}|u_\ell-\tilde u_\ell|+\sum_{\ell=0}^{k-3}|v_\ell-\tilde v_\ell|
+|\nabla u_{k-2}-\nabla\tilde u_{k-2}| \Big ).
\end{align}
\end{Lem}
\begin{proof}
Suppose we have two polyharmonic maps \(\varphi,\tilde\varphi\) of order \(k=2s\). Then, from \eqref{definition-Fk-even}, we get
\begin{align*}
(F^k)^\alpha-(\tilde F^k)^\alpha=&-A^\alpha_{k-1}+\tilde A^\alpha_{k-1} \\
&-u^\delta_{k-2}\langle d\varphi^\gamma,d\varphi^\beta\rangle R^\alpha_{\beta\gamma\delta}(\varphi)
+\tilde u^\delta_{k-2}\langle d\tilde\varphi^\gamma,d\tilde\varphi^\beta\rangle R^\alpha_{\beta\gamma\delta}(\tilde\varphi) \\
&-\sum_{\ell=1}^{\frac{k}{2}-1}\big(u^\delta_{\frac{k}{2}-\ell-1}\langle v^\gamma_{\frac{k}{2}+\ell-2},d\varphi^\beta\rangle R^\alpha_{\beta\gamma\delta}(\varphi)
-\tilde u^\delta_{\frac{k}{2}-\ell-1}\langle \tilde v^\gamma_{\frac{k}{2}+\ell-2},d\tilde\varphi^\beta\rangle R^\alpha_{\beta\gamma\delta}(\tilde\varphi) \big) \\
&-\sum_{\ell=1}^{\frac{k}{2}-1}\big(u^\vartheta_{\frac{k}{2}+\ell-2}u^\delta_{\frac{k}{2}-\ell-1}\langle d\varphi^\eta,d\varphi^\beta\rangle E^\alpha_{\beta\delta\vartheta\eta}(\varphi) 
-\tilde u^\vartheta_{\frac{k}{2}+\ell-2}\tilde u^\delta_{\frac{k}{2}-\ell-1}\langle d\tilde\varphi^\eta,d\tilde\varphi^\beta\rangle E^\alpha_{\beta\delta\vartheta\eta}(\tilde\varphi) \big) \\
&-\sum_{\ell=1}^{\frac{k}{2}-1}\big(u^\delta_{\frac{k}{2}+\ell-2}\langle v^\gamma_{\frac{k}{2}-\ell-1},d\varphi^\beta\rangle R^\alpha_{\beta\gamma\delta}(\varphi)
-\tilde u^\delta_{\frac{k}{2}+\ell-2}\langle \tilde v^\gamma_{\frac{k}{2}-\ell-1},d\tilde\varphi^\beta\rangle R^\alpha_{\beta\gamma\delta}(\tilde\varphi)\big).
\end{align*}
In order to estimate the first term on the right hand side we use \eqref{estimate-A}. It is straightforward to estimate
\begin{align*}
|u^\delta_{k-2}\langle d\varphi^\gamma,d\varphi^\beta\rangle R^\alpha_{\beta\gamma\delta}(\varphi)
-\tilde u^\delta_{k-2}\langle d\tilde\varphi^\gamma,d\tilde\varphi^\beta\rangle R^\alpha_{\beta\gamma\delta}(\tilde\varphi)|
\leq C(|\varphi-\tilde\varphi|+|d\varphi-d\tilde\varphi|+|u_{k-2}-\tilde u_{k-2}|),
\end{align*}
which controls the second term on the right hand side. Concerning the first term inside the sum we rewrite
\begin{align*}
u^\delta_{\frac{k}{2}-\ell-1}&\langle v^\gamma_{\frac{k}{2}+\ell-2},d\varphi^\beta\rangle R^\alpha_{\beta\gamma\delta}(\varphi)
-\tilde u^\delta_{\frac{k}{2}-\ell-1}\langle \tilde v^\gamma_{\frac{k}{2}+\ell-2},d\tilde\varphi^\beta\rangle R^\alpha_{\beta\gamma\delta}(\tilde\varphi) \\
=&(u^\delta_{\frac{k}{2}-\ell-1}
-\tilde u^\delta_{\frac{k}{2}-\ell-1})\langle v^\gamma_{\frac{k}{2}+\ell-2},d\varphi^\beta\rangle R^\alpha_{\beta\gamma\delta}(\varphi)
+\tilde u^\delta_{\frac{k}{2}-\ell-1}\langle v^\gamma_{\frac{k}{2}+\ell-2}-\tilde v^\gamma_{\frac{k}{2}+\ell-2},d\varphi^\beta\rangle R^\alpha_{\beta\gamma\delta}(\varphi) \\
&+\tilde u^\delta_{\frac{k}{2}-\ell-1}\langle\tilde v^\gamma_{\frac{k}{2}+\ell-2},d\varphi^\beta-\tilde d\varphi^\beta\rangle R^\alpha_{\beta\gamma\delta}(\varphi)
+\tilde u^\delta_{\frac{k}{2}-\ell-1}\langle\tilde v^\gamma_{\frac{k}{2}+\ell-2},\tilde d\varphi^\beta\rangle 
\big(R^\alpha_{\beta\gamma\delta}(\varphi)-R^\alpha_{\beta\gamma\delta}(\tilde\varphi)\big).
\end{align*}
Hence, we deduce the estimate
\begin{align*}
|u^\delta_{\frac{k}{2}-\ell-1}&\langle v^\gamma_{\frac{k}{2}+\ell-2},d\varphi^\beta\rangle R^\alpha_{\beta\gamma\delta}(\varphi)
-\tilde u^\delta_{\frac{k}{2}-\ell-1}\langle \tilde v^\gamma_{\frac{k}{2}+\ell-2},d\tilde\varphi^\beta\rangle R^\alpha_{\beta\gamma\delta}(\tilde\varphi)| \\
&\leq C(|\varphi-\tilde\varphi|+|d\varphi-d\tilde\varphi|+|u_{\frac{k}{2}-\ell-1}-\tilde u_{\frac{k}{2}-\ell-1}|
+|v_{\frac{k}{2}+\ell-2}-\tilde v_{\frac{k}{2}+\ell-2}|).
\end{align*}

Regarding the second term in the sum we get
\begin{align*}
u^\vartheta_{\frac{k}{2}+\ell-2}&u^\delta_{\frac{k}{2}-\ell-1}\langle d\varphi^\eta,d\varphi^\beta\rangle E^\alpha_{\beta\delta\vartheta\eta}(\varphi) 
-\tilde u^\vartheta_{\frac{k}{2}+\ell-2}\tilde u^\delta_{\frac{k}{2}-\ell-1}\langle d\tilde\varphi^\eta,d\tilde\varphi^\beta\rangle E^\alpha_{\beta\delta\vartheta\eta}(\tilde\varphi) \\
=&(u^\vartheta_{\frac{k}{2}+\ell-2}-\tilde u^\vartheta_{\frac{k}{2}+\ell-2})u^\delta_{\frac{k}{2}-\ell-1}\langle d\varphi^\eta,d\varphi^\beta\rangle E^\alpha_{\beta\delta\vartheta\eta}(\varphi) \\
&+\tilde u^\vartheta_{\frac{k}{2}+\ell-2}(u^\delta_{\frac{k}{2}-\ell-1}-\tilde u^\delta_{\frac{k}{2}-\ell-1})\langle d\varphi^\eta,d\varphi^\beta\rangle E^\alpha_{\beta\delta\vartheta\eta}(\varphi)\\
&+\tilde u^\vartheta_{\frac{k}{2}+\ell-2}\tilde u^\delta_{\frac{k}{2}-\ell-1}\langle d\varphi^\eta-d\tilde\varphi^\eta,d\varphi^\beta\rangle E^\alpha_{\beta\delta\vartheta\eta}(\varphi)
+\tilde u^\vartheta_{\frac{k}{2}+\ell-2}\tilde u^\delta_{\frac{k}{2}-\ell-1}\langle d\tilde\varphi^\eta,d\varphi^\beta-d\tilde\varphi^\beta\rangle E^\alpha_{\beta\delta\vartheta\eta}(\varphi)\\
&+\tilde u^\vartheta_{\frac{k}{2}+\ell-2}\tilde u^\delta_{\frac{k}{2}-\ell-1}\langle d\tilde\varphi^\eta,d\tilde\varphi^\beta\rangle 
\big(E^\alpha_{\beta\delta\vartheta\eta}(\varphi) - E^\alpha_{\beta\delta\vartheta\eta}(\tilde\varphi)\big)
\end{align*}
and obtain the estimate
\begin{align*}
|u^\vartheta_{\frac{k}{2}+\ell-2}&u^\delta_{\frac{k}{2}-\ell-1}\langle d\varphi^\eta,d\varphi^\beta\rangle E^\alpha_{\beta\delta\vartheta\eta}(\varphi) 
-\tilde u^\vartheta_{\frac{k}{2}+\ell-2}\tilde u^\delta_{\frac{k}{2}-\ell-1}\langle d\tilde\varphi^\eta,d\tilde\varphi^\beta\rangle E^\alpha_{\beta\delta\vartheta\eta}(\tilde\varphi)| \\
&\leq C\big(|u_{\frac{k}{2}+\ell-2}-\tilde u_{\frac{k}{2}+\ell-2}|+|u_{\frac{k}{2}-\ell-1}-\tilde u_{\frac{k}{2}-\ell-1}|+|d\varphi-d\tilde\varphi|+|\varphi-\tilde\varphi|\big).
\end{align*}
The last term in the sum may be rewritten as
\begin{align*}
u^\delta_{\frac{k}{2}+\ell-2}&\langle v^\gamma_{\frac{k}{2}-\ell-1},d\varphi^\beta\rangle R^\alpha_{\beta\gamma\delta}(\varphi)
-\tilde u^\delta_{\frac{k}{2}+\ell-2}\langle \tilde v^\gamma_{\frac{k}{2}-\ell-1},d\tilde\varphi^\beta\rangle R^\alpha_{\beta\gamma\delta}(\tilde\varphi) \\
=&(u^\delta_{\frac{k}{2}+\ell-2}
-\tilde u^\delta_{\frac{k}{2}+\ell-2})\langle v^\gamma_{\frac{k}{2}-\ell-1},d\varphi^\beta\rangle R^\alpha_{\beta\gamma\delta}(\varphi)
+\tilde u^\delta_{\frac{k}{2}+\ell-2}\langle v^\gamma_{\frac{k}{2}-\ell-1}
-\tilde v^\gamma_{\frac{k}{2}-\ell-1},d\varphi^\beta\rangle R^\alpha_{\beta\gamma\delta}(\varphi) \\
&+\tilde u^\delta_{\frac{k}{2}+\ell-2}\langle \tilde v^\gamma_{\frac{k}{2}-\ell-1},d\varphi^\beta
-d\tilde\varphi^\beta\rangle R^\alpha_{\beta\gamma\delta}(\varphi)
+\tilde u^\delta_{\frac{k}{2}+\ell-2}\langle \tilde v^\gamma_{\frac{k}{2}-\ell-1},d\tilde\varphi^\beta\rangle 
\big(R^\alpha_{\beta\gamma\delta}(\varphi)-R^\alpha_{\beta\gamma\delta}(\tilde\varphi)\big)
\end{align*}
which leads us to the estimate
\begin{align*}
|u^\delta_{\frac{k}{2}+\ell-2}&\langle v^\gamma_{\frac{k}{2}-\ell-1},d\varphi^\beta\rangle R^\alpha_{\beta\gamma\delta}(\varphi)
-\tilde u^\delta_{\frac{k}{2}+\ell-2}\langle \tilde v^\gamma_{\frac{k}{2}-\ell-1},d\tilde\varphi^\beta\rangle R^\alpha_{\beta\gamma\delta}(\tilde\varphi)| \\
&\leq C\big(|u_{\frac{k}{2}+\ell-2}-\tilde u_{\frac{k}{2}+\ell-2}|+|v_{\frac{k}{2}-\ell-1}-\tilde v_{\frac{k}{2}-\ell-1}|
+|d\varphi-d\tilde\varphi|+|\varphi-\tilde\varphi|\big).
\end{align*}
The claim now follows by also using \eqref{estimate-A}, where now we have to write $|\nabla u_{k-2}- \nabla \tilde u_{k-2}|$ instead of $|v_{k-2}- \tilde v_{k-2}|$ since these last functions do not appear in the definition of $z$.
\end{proof}

Combining the inequalities \eqref{estimate-Delta-varphi}, \eqref{estimate-differential}, \eqref{estimate-A}, \eqref{estimate-dA} and \eqref{estimate-Fk}
we arrive at
\begin{align}
\label{estimate-full-Deltaz}
|\Delta z|\leq C\big(&|\varphi-\tilde\varphi|+|d\varphi-d\tilde\varphi|+|\nabla d\varphi-\nabla d\tilde\varphi| \\
\nonumber&+\sum_{\ell=0}^{k-2}|u_\ell-\tilde u_\ell|
+\sum_{\ell=0}^{k-3}(|v_\ell-\tilde v_\ell|+|\nabla v_\ell-\nabla\tilde v_\ell|)
+|\nabla u_{k-2}-\nabla\tilde u_{k-2}|
\big).
\end{align}
Due to the estimate \eqref{estimate-full-Deltaz} the assumptions of Theorem \ref{aro-theorem} are satisfied
and we can conclude that \(z=0\) which in particular implies that \(\varphi=\tilde\varphi\).
To complete the proof we make use of the same globalization argument which was employed in \cite[Proof of Theorem~1.3]{MR3990379}.
\end{proof}

\begin{proof}[Proof of Theorem \ref{theorem-totally-geodesic-sphere}] 
The proof is again based on Aronszajn's Theorem \ref{aro-theorem} and the explicit
expressions of the Christoffel symbols on \(\mathbb{S}^n\).

Let \(\mathbb{S}^n\) be the Euclidean unit sphere and denote by \(N\) and \(S\)
the north and south pole, respectively.
It is well-known that
\begin{align*}
\mathbb{S}^n\setminus\{N,S\}=\left(\mathbb{S}^{n-1}\times (0,\pi),\sin^2 s \cdot g_{\mathbb{S}^{n-1}}+ds^2\right).
\end{align*}

Let \((y^a)\) be local coordinates on \(\mathbb{S}^{n-1},a=1,\ldots,n-1\).
Then \(y=(y^1,\ldots,y^{n-1},y^n=s)=(\tilde{y},s)\) are local coordinates on \(\mathbb{S}^n\setminus\{N,S\}\).

In this geometric setup the Christoffel symbols on \(\mathbb{S}^n\) are given by

\begin{equation}
\label{christoffel-sphere-equator}
\begin{split}
\Gamma^a_{bc}(y)=&\tilde\Gamma^a_{bc}(\tilde{y}),\quad a,b,c=1,\ldots,n-1,\\
\Gamma^n_{bc}(y)=&-\sin s\cos s \ \tilde g_{bc}(\tilde{y})=-\frac{1}{2}\sin(2s) \ \tilde g_{bc}(\tilde{y}),\\
\Gamma^a_{nn}(y)=&\Gamma^n_{nn}(y)=\Gamma^n_{bn}(y)=0,\\
\Gamma^a_{bn}(y)=&\frac{\cos s}{\sin s}\delta^a_b,\\
\end{split}
\end{equation}
where we use a ``\;\(\tilde{}\)\;'' to indicate objects on \(\mathbb{S}^{n-1}\).
The equator \(\mathbb{S}^{n-1}\) is given by
\begin{align*}
\mathbb{S}^{n-1}:\quad  y^n=s=\frac{\pi}{2}.
\end{align*}

In the special case that $N=\s^n$ we know that:
\begin{equation}
\label{tensor-curvature-s^n}
R(X,Y)Z= -\langle X,Z \rangle Y+\langle Y,Z \rangle X \quad  \forall \,X,Y,Z \in C(T\s^n) \,.
\end{equation}
Next, using \eqref{tensor-curvature-s^n}, we express $R^\alpha_{\delta \beta \gamma}$ with respect to our local coordinates \((\tilde{y},s)\). For our purposes we only need to compute  explicitly $R^n_{\delta \beta \gamma}$. Using \eqref{Convention-R} and \eqref{tensor-curvature-s^n}, we find that the only non-zero terms of this type are:
\begin{equation}\label{tensor-curvature-sn-last-component}
 - R^n_{ a n b}= R^n_{ a b n}=\langle R\left (\frac{\partial}{\partial y^b},\frac{\partial}{\partial s} \right ) \frac{\partial}{\partial y^a},\frac{\partial}{\partial s}  \rangle=-(\sin  s)^2\,\tilde{g}_{b a}(\tilde{y})\,.\\
  \end{equation} 

Now, let \((U,x^i)\) be a local chart on \(M\) and denote the domain of the above local coordinates on \(\mathbb{S}^n\setminus\{N,S\}\) by \(V\).
In addition, we assume that \(\varphi(U)\subset V\).

For simplicity, we again denote the expression of \(\varphi\) in local coordinates also by \(\varphi\), i.e.
\begin{align*}
\varphi=(\varphi^1,\ldots,\varphi^n).
\end{align*}

Assume that \(W\) is an open subset of \(U\) and \(\varphi(W)\subset \mathbb{S}^{n-1}\), i.e. \(\varphi(W)\subset \mathbb{S}^{n-1}\cap V\).
Hence, in \(W\) we have \(\varphi^n=\frac{\pi}{2}\). Now, define \(f\colon U\to\R, f:=\varphi^n-\frac{\pi}{2}\).
Clearly, the function \(f\) vanishes when restricted to \(W\).

Let \(D\) be an open subset of \(U\) such that its closure in \(M\)
is compact and included in $U$, and \(W\subset D\subset U\).

We define a vector-valued function $y$ for the \(n\)-th component of the map \(\varphi\) as follows
\begin{equation}\label{def-u-sn}
y=\left ( \begin{array}{l}
f\\
df \\
u_0^n\\
v_0^n\\
u_1^n\\
v_1^n\\
\vdots \\
v^n_{k-3}\\
u^n_{k-2}
\end{array} \right )\,,
\end{equation}
where
$u_0^n,\ldots,u_{k-2}^n$ and $v_0^n,\ldots,v_{k-3}^n$ are the $n$-components of the variables $u_0,\ldots,u_{k-2}$ and $v_0,\ldots,v_{k-3}$ defined as in the proof of Theorem~\ref{Th-harmonic-everywhere}. In addition, we define 
\begin{equation}\label{def-F}
F=\left ( \begin{array}{l}
\Delta f\\
\Delta (df)\\
u_1^n-A_1^n\\
d(u_1^n-A_1^n)\\
u_2^n-A_2^n\\
d(u_2^n-A_2^n)\\
\vdots\\
d(u^n_{k-2}-A^n_{k-2})\\
(F^k)^n
\end{array} \right ) \,.
\end{equation}

Again, by construction 
$$
\Delta y =F+\textrm{~terms linear in~} v^n_j,~~ \nabla v^n_j, \qquad j=0,\ldots, k-3.
$$ 

In the following we will give the proof of Theorem \ref{theorem-totally-geodesic-sphere}.
We will only consider the case of a polyharmonic map of even order as the odd case
can be treated by exactly the same methods.

\begin{Lem}
\label{lemma-sampson-estimate-A}
Let \(\varphi\colon M\to\s^n\) be a map with corresponding variables \eqref{def-u-sn}. Assume that \(0 \leq j\leq k-3\).
Then the following estimates hold on $D$
\begin{equation}
\label{estimate-A-sampson-sn}
\begin{split}
|\Delta f|\leq & C(|f|+|u_0^n|), \\
|\Delta (df)|\leq & C(|f|+|df|+|\nabla df|+|v_0^n|), \\
|A^n_{j+1}|\leq & C(|f|+|df|)\,.
\end{split}
\end{equation}
\end{Lem}
\begin{proof}
Recall that 
\begin{align*}
\Delta f=\Delta\varphi^n=-u_0^n+\langle d\varphi^\beta,d\varphi^\nu\rangle\Gamma^n_{\beta\nu}.
\end{align*}
Using the explicit form of the Christoffel symbols \eqref{christoffel-sphere-equator} we obtain the following expansion
\begin{align*}
\langle d\varphi^\beta,d\varphi^\nu\rangle\Gamma^n_{\beta\nu}=
\langle d\varphi^n,d\varphi^\nu\rangle\underbrace{\Gamma^n_{n\nu}}_{=0}
+\langle d\varphi^b,d\varphi^c\rangle\Gamma^n_{bc}
+\langle d\varphi^b,d\varphi^n\rangle\underbrace{\Gamma^n_{bn}}_{=0}
=\langle d\varphi^b,d\varphi^c\rangle\Gamma^n_{bc}\,.
\end{align*}
As we have
\begin{align*}
\left|\Gamma^n_{bc}\right| = & \left|-\frac{1}{2}\sin(2\varphi^n)\tilde{g}_{bc} \right| = \left|\frac{1}{2}\sin(2f)\tilde{g}_{bc} \right| \\
\nonumber \leq & C|f|
\end{align*}
we obtain
\begin{align*}
|\langle d\varphi^\beta,d\varphi^\nu\rangle\Gamma^n_{\beta\nu}|\leq C|f|
\end{align*}
establishing the first estimate. 

Now, remember that
\begin{align}\label{def-An}
A_{j+1}^n=&
-2\langle du_{j}^\nu,d\varphi^\beta\rangle\Gamma^n_{\beta\nu}
-u_{j}^\nu u_0^\beta\Gamma_{\beta\nu}^n
+u_{j}^\nu\langle d\varphi^\vartheta,d\varphi^\sigma\rangle C_{\vartheta\sigma\nu}^n,
\end{align}
where 
\begin{align*}
C_{\vartheta\sigma\nu}^n:=&
\Gamma^\mu_{\vartheta\sigma}\Gamma^n_{\mu\nu}
-\frac{\partial \Gamma_{\vartheta \nu}^n}{\partial y^\sigma} - \Gamma_{\vartheta \nu}^\gamma\, \Gamma_{\sigma \gamma}^n.
\end{align*}

Using again the explicit form of the Christoffel symbols \eqref{christoffel-sphere-equator} we find the estimates
\begin{align*}
|\langle du_{j}^\beta,d\varphi^\nu\rangle\Gamma^n_{\beta\nu}|&\leq C|f|,\\
|u_{j}^\nu u_0^\beta\Gamma_{\beta\nu}^n|&\leq C|f|.
\end{align*}
As for the terms proportional to \(C_{\vartheta\sigma\nu}^\alpha\), we note that
\begin{align*}
u_{j}^\nu\langle d\varphi^\vartheta,d\varphi^\sigma\rangle\Gamma^\mu_{\vartheta\sigma}\Gamma^n_{\mu\nu}
=u_{j}^c\langle d\varphi^\vartheta,d\varphi^\sigma\rangle\Gamma^b_{\vartheta\sigma}\Gamma^n_{bc}.
\end{align*}
This allows us to derive the estimate
\begin{align*}
|u_{j}^\nu\langle d\varphi^\vartheta,d\varphi^\sigma\rangle\Gamma^\mu_{\vartheta\sigma}\Gamma^n_{\mu\nu}|
\leq C|f|.
\end{align*}
In order to estimate \(u_{j}^\nu\langle d\varphi^\vartheta,d\varphi^\sigma\rangle\frac{\partial \Gamma_{\vartheta \nu}^n}{\partial y^\sigma}\)
we make use of the same strategy as before taking into account that
\begin{align*}
\frac{\partial \Gamma_{ab}^n}{\partial y^c}=-\frac{1}{2}\sin(2f)\frac{\partial\tilde g_{ab}}{\partial y^c}
\end{align*}
which follows from \eqref{christoffel-sphere-equator}.

Since the last two terms are proportional to \(C_{\vartheta\sigma\nu}^\alpha\), they can then be estimated as
\begin{align*}
|u_{j}^\nu\langle d\varphi^\vartheta,d\varphi^\sigma\rangle\frac{\partial \Gamma_{\vartheta \nu}^n}{\partial y^\sigma}|
&\leq C(|f|+|df|), \\
|u_{j}^\nu\langle d\varphi^\vartheta,d\varphi^\sigma\rangle\Gamma_{\vartheta \nu}^\gamma\, \Gamma_{\sigma \gamma}^n|
&\leq C|f|.
\end{align*}
This proves the third estimate of the lemma. The estimate on \(\Delta(df)\) can be achieved by 
exactly the same methods.
\end{proof}

\begin{Lem}
Let \(\varphi\colon M\to\s^n\) be a map with corresponding variables \eqref{def-u-sn}. Assume that \(0 \leq j\leq k-3\).
Then the following estimate holds
\begin{align}
\label{estimate-dA-sampson-sn}
|dA^n_{j+1}|\leq C(|f|+|df|+|\nabla df|)\,.
\end{align}
\end{Lem}

\begin{proof}
By a direct calculation we find
\begin{align*}
\nabla_i A^n_{j+1}=&-2\langle\nabla_iv_j^\gamma,d\varphi^\beta\rangle\Gamma^n_{\beta\gamma}
-2\langle v_j^\gamma,\nabla_i d\varphi^\beta\rangle\Gamma^n_{\beta\gamma}
-2\langle v_j^\gamma,d\varphi^\beta\rangle\frac{\partial\Gamma^n_{\beta\gamma}}{\partial y^\delta}\varphi^\delta_i \\
&-\nabla_iu_j^\gamma\, u_0^\beta\,\Gamma^n_{\beta\gamma}-u_j^\gamma \,\nabla_iu_0^\beta\,\Gamma^n_{\beta\gamma}
-u_j^\gamma\, u_0^\beta\,\frac{\partial\Gamma^n_{\beta\gamma}}{\partial y^\delta}\varphi^\delta_i \\
&+\nabla_i u^\nu_j\langle d\varphi^\vartheta,d\varphi^\sigma\rangle C^n_{\vartheta\sigma\nu}
+ u^\nu_j\langle\nabla_i d\varphi^\vartheta,d\varphi^\sigma\rangle C^n_{\vartheta\sigma\nu}
+ u^\nu_j\langle d\varphi^\vartheta,\nabla_id\varphi^\sigma\rangle C^n_{\vartheta\sigma\nu} \\
&+ u^\nu_j\langle d\varphi^\vartheta,d\varphi^\sigma\rangle \frac{\partial C^n_{\vartheta\sigma\nu}}{\partial y^\delta}\varphi^\delta_i.
\end{align*}
All the terms in this expression can be estimated by expanding the Christoffel symbols and using the same strategy as in the proof of Lemma~\ref{lemma-sampson-estimate-A} except the last term which requires a more careful inspection.

Again, a direct calculation yields
\begin{align*}
\frac{\partial C^n_{\vartheta\sigma\nu}}{\partial y^\delta}=
\frac{\partial \Gamma^\mu_{\vartheta\sigma}}{\partial y^\delta}
\Gamma^n_{\mu\nu}+
\Gamma^\mu_{\vartheta\sigma}
\frac{\partial \Gamma^n_{\mu\nu}}{\partial y^\delta}
-\frac{\partial^2\Gamma_{\vartheta \nu}^n}{\partial y^\sigma\partial y^\delta}
-\frac{\partial\Gamma_{\vartheta \nu}^\gamma}{\partial y^\delta}\Gamma_{\sigma \gamma}^n
-\Gamma_{\vartheta \nu}^\gamma\frac{\partial\Gamma_{\sigma \gamma}^n}{\partial y^\delta}.
\end{align*}
We realize that all terms in 
\(u^\nu_j\langle d\varphi^\vartheta,d\varphi^\sigma\rangle \frac{\partial C^n_{\vartheta\sigma\nu}}{\partial y^\delta}\varphi^\delta_i\)
can be estimated by the same reasoning used before, only the contribution that is proportional 
to the second derivative of the Christoffel symbols needs to be treated in more detail. 
Hence, let us have a closer look at
\begin{align*}
\frac{\partial^2\Gamma_{\vartheta \nu}^n}{\partial y^\sigma\partial y^\delta}
u^\nu_j\langle d\varphi^\vartheta,d\varphi^\sigma\rangle \varphi^\delta_i
=&
\frac{\partial^2\Gamma_{ab}^n}{\partial y^c\partial y^d}
u^b_j\langle d\varphi^a,d\varphi^c\rangle \varphi^d_i 
+
\frac{\partial^2\Gamma_{ab}^n}{\partial y^n\partial y^d}
u^b_j\langle d\varphi^a,d\varphi^n\rangle \varphi^d_i \\
&+
\frac{\partial^2\Gamma_{ab}^n}{\partial y^c\partial y^n}
u^b_j\langle d\varphi^a,d\varphi^c\rangle \varphi^n_i 
+
\frac{\partial^2\Gamma_{ab}^n}{\partial y^n\partial y^n}
u^b_j\langle d\varphi^a,d\varphi^n\rangle \varphi^n_i\,.
\end{align*}
Using \eqref{christoffel-sphere-equator} we obtain
\begin{align*}
\frac{\partial^2\Gamma_{ab}^n}{\partial y^c\partial y^d}=-\frac{1}{2}\sin(2f)\frac{\partial^2\tilde g_{ab}}{\partial y^c\partial y^d}
\end{align*}
and we can conclude that all terms can be estimated in such a way that the statement of the lemma holds true.
\end{proof}

\begin{Lem}
Let \(\varphi\colon M\to\s^n\) be a polyharmonic map of even order with corresponding variables \eqref{def-u-sn}.
Then the following estimate holds
\begin{align}
\label{estimate-Fk-sampson-sn}
|(F^k)^n|\leq C\big(|f|+|df|+\sum_{\ell=0}^{k-2}|u^n_\ell|+\sum_{\ell=0}^{k-3}|v^n_\ell|+|\nabla u^n_{k-2}|\big).
\end{align}
\end{Lem}
\begin{proof}

Recall that
\begin{align*}
(F^k)^n=&-A^n_{k-1}
-u^\delta_{k-2}\langle d\varphi^\gamma,d\varphi^\beta\rangle R^n_{\beta\gamma\delta} \\
\nonumber&-\sum_{\ell=1}^{\frac{k}{2}-1}\big(u^\delta_{\frac{k}{2}-\ell-1}\langle v^\gamma_{\frac{k}{2}+\ell-2},d\varphi^\beta\rangle R^n_{\beta\gamma\delta}
+u^\vartheta_{\frac{k}{2}+\ell-2}u^\delta_{\frac{k}{2}-\ell-1}\langle d\varphi^\eta,d\varphi^\beta\rangle E^n_{\beta\delta\vartheta\eta} \\
\nonumber &\hspace{1.5cm} +u^\delta_{\frac{k}{2}+\ell-2}\langle v^\gamma_{\frac{k}{2}-\ell-1},d\varphi^\beta\rangle R^n_{\beta\gamma\delta}\big), 
\end{align*}
where
\begin{align*}
E^n_{\beta\delta\vartheta\eta}:=&R^n_{\beta\gamma\delta}\Gamma^\gamma_{\vartheta\eta}+R^n_{\beta\gamma\vartheta}\Gamma^\gamma_{\delta\eta}.
\end{align*}
First, we use the third inequality in \eqref{estimate-A-sampson-sn} to estimate \(A^n_{k-1}\).
It remains to estimate the terms involving the curvature tensor.
Inserting the non-zero components of the curvature tensor given in \eqref{tensor-curvature-sn-last-component}
we find
\begin{align*}
u^\delta_{k-2}\langle d\varphi^\gamma,d\varphi^\beta\rangle R^n_{\beta\gamma\delta}
=u^\delta_{k-2}\langle d\varphi^n,d\varphi^\beta\rangle R^n_{\beta n\delta}
+u^n_{k-2}\langle d\varphi^\gamma,d\varphi^\beta\rangle R^n_{\beta\gamma n}.
\end{align*}
Hence, we may estimate
\begin{align*}
|u^\delta_{k-2}\langle d\varphi^\gamma,d\varphi^\beta\rangle R^n_{\beta\gamma\delta}|\leq 
C\big (|df|+|u^n_{k-2}| \big ).
\end{align*}
By the same reasoning we find 
\begin{align*}
|u^\delta_{\frac{k}{2}-\ell-1}\langle v^\gamma_{\frac{k}{2}+\ell-2},d\varphi^\beta\rangle R^n_{\beta\gamma\delta}|
&\leq C \big (|u^n_{\frac{k}{2}-\ell-1}|+|v^n_{\frac{k}{2}+\ell-2}| \big ),\\
|u^\delta_{\frac{k}{2}+\ell-2}\langle v^\gamma_{\frac{k}{2}-\ell-1},d\varphi^\beta\rangle R^n_{\beta\gamma\delta}|
&\leq C \big (|u^n_{\frac{k}{2}+\ell-2}|+|v^n_{\frac{k}{2}-\ell-1}| \big ).
\end{align*}
Regarding the term that is proportional to \(E^n_{\beta\delta\vartheta\eta}\) we expand
\begin{align*}
u^\vartheta_{\frac{k}{2}+\ell-2}u^\delta_{\frac{k}{2}-\ell-1}\langle d\varphi^\eta,d\varphi^\beta\rangle E^n_{\beta\delta\vartheta\eta}
=& \,
u^b_{\frac{k}{2}+\ell-2}u^\delta_{\frac{k}{2}-\ell-1}\langle d\varphi^c,d\varphi^\beta\rangle R^n_{\beta n\delta}\Gamma^n_{bc}\\
&+u^\vartheta_{\frac{k}{2}+\ell-2}u^n_{\frac{k}{2}-\ell-1}\langle d\varphi^\eta,d\varphi^\beta\rangle R^n_{\beta\gamma n}\Gamma^\gamma_{\vartheta\eta}\\
&+u^\vartheta_{\frac{k}{2}+\ell-2}u^b_{\frac{k}{2}-\ell-1}\langle d\varphi^c,d\varphi^\beta\rangle R^n_{\beta n\vartheta}\Gamma^n_{bc}\\
&+u^n_{\frac{k}{2}+\ell-2}u^\delta_{\frac{k}{2}-\ell-1}\langle d\varphi^\eta,d\varphi^\beta\rangle R^n_{\beta\gamma n}\Gamma^\gamma_{\delta\eta},
\end{align*}
where we used both \eqref{christoffel-sphere-equator} and \eqref{tensor-curvature-sn-last-component}.
Hence, we can infer the estimate
\begin{align*}
|u^\vartheta_{\frac{k}{2}+\ell-2}u^\delta_{\frac{k}{2}-\ell-1}\langle d\varphi^\eta,d\varphi^\beta\rangle E^n_{\beta\delta\vartheta\eta}|
\leq C\big (|f|+|u^n_{\frac{k}{2}+\ell-2}|+|u^n_{\frac{k}{2}-\ell-1}| \big ).
\end{align*}
The claim now follows by combining all the single estimates.
\end{proof}

Now, we show that the vector variable $y$ defined in \eqref{def-u-sn} satisfies
\begin{equation}\label{y=0}
y \equiv 0 \quad {\rm on}\,\, W\,.
\end{equation}
To prove this claim, we observe that $f=\varphi^n -\pi/2$, and consequently $df$, vanish because $\varphi$ maps $W$ into the equator. Next, using $\Delta f=0$ on $W$ and the explicit expression \eqref{christoffel-sphere-equator} of the Christoffel symbols we deduce that $u^n_0=0$. The functions $A^n_{j}$, $j=1, \ldots,k-2$ also vanish identically on $W$. This follows easily from the definition \eqref{def-An}, using again the explicit expression \eqref{christoffel-sphere-equator} of the Christoffel symbols together with $d\varphi^n=0$. Finally, from these facts it is easy to deduce that all the components of $y$ vanish on $W$ and so the claim \eqref{y=0} holds. 
 
Next, using the inequalities \eqref{estimate-A-sampson-sn}, \eqref{estimate-dA-sampson-sn} and \eqref{estimate-Fk-sampson-sn}
we find
\begin{align}
\label{estimate-full-Deltay}
|\Delta y|\leq C\big(|f|+|df|+|\nabla df|+\sum_{\ell=0}^{k-2}|u^n_\ell|+\sum_{\ell=0}^{k-3}(|v^n_\ell|+|\nabla v^n_\ell|)
+|\nabla u^n_{k-2}|\big).
\end{align}
Now, because of \eqref{y=0} and the estimate \eqref{estimate-full-Deltay}, the assumptions of Theorem \ref{aro-theorem} are satisfied
and we can conclude that \(y=0\) on $D$, which in particular implies that \(\varphi^n=\frac{\pi}{2}\),
i.e. $\varphi$ maps the whole of $D$ into $\mathbb{S}^{n-1}$.
We finish the proof of Theorem~\ref{theorem-totally-geodesic-sphere} by setting $A:=\{p\in M : \varphi(p)\in \mathbb{S}^{n-1}\}$ and using the same globalization argument as above.
\end{proof}

\begin{Bem} Theorems \ref{Th-harmonic-everywhere}, \ref{main-theorem} and \ref{theorem-totally-geodesic-sphere} have some important applications. For instance, these results can be used to simplify proofs and obtain uniqueness and reduction results for $k$-harmonic submanifolds in spheres. For simplicity, we shall illustrate these facts in the case that $k=2$:
\begin{itemize}
\item[(a)] A result of J.H. Chen (see \cite{JHChen}) says that a compact proper biharmonic hypersurface $M^m$ in $\s^{m+1}$ with $|A|^2 \leq m$ is CMC (and thus $|A|^2 = m$). The proof (in Chinese) is long and skillful, but it can be simplified using Theorem \ref{Th-harmonic-everywhere}. This simplified proof was given in \cite{MR2945959,SMCO}.

\item[(b)] All CMC proper biharmonic immersions from $\R^2$ in $\s^m$ are given in \cite{MR3523535}. The uniqueness part of this result follows from the fact that on an open subset of $\R^2$ such immersions must have a certain form and by Theorem \ref{main-theorem} their extensions to $\R^2$ are unique.

\item[(c)] Theorem \ref{theorem-totally-geodesic-sphere}, as an alternative to Theorem~\ref{Th-harmonic-everywhere}, could be used in the final argument of the proof of Theorem~1 in \cite{Vieira-arxiv}: since an open subset of $M^m$ lies in $\s^m$, the equator of $\s^{m+1}$, the whole of $M^m$ lies in $\s^m$, i.e. $M^m=\s^m$.

\noindent Also, Theorem \ref{theorem-totally-geodesic-sphere} could be useful to obtain reduction results as the first normal bundle does not need to be defined on the whole of $M$, but only on an open subset of it.
\end{itemize}

\end{Bem}

\section{Unique continuation theorems for ES-$4$-harmonic maps}
\label{section-es4}
In this section we study unique continuation properties for critical points
of the ES-$4$-energy \eqref{k-energy-es}. We shall prove two unique continuation results for $k=4$, which is the only case for which we know the explicit form of the Euler-Lagrange equations (see \cite{MR4106647}). 

The energy functional for ES-$4$-harmonic maps (corresponding to \eqref{k-energy-es} with \(k=4\)) is given by
\begin{align}
\label{energy-es4-harmonic}
E^{ES}_4(\varphi)&=\frac{1}{2}\int_M|(d+d^\ast)^4\varphi|^2\dv\\
\nonumber&=\frac{1}{2}\int_M|\bar\Delta\tau(\varphi)|^2\dv
+\frac{1}{4}\int_M|R^N(d\varphi(e_i),d\varphi(e_j))\tau(\varphi)|^2\dv.
\end{align}

Note that here and in the sequel we shall omit to write the symbol $\sum$ when it is clear from the context.

The first variation of \eqref{energy-es4-harmonic} 
was calculated in \cite[Section 3]{MR4106647} and is characterized by the vanishing 
of the ES-4-tension field \(\tau_4^{ES}(\varphi)\) given by the following expression
\begin{align}
\label{es-4-tension}
\tau_4^{ES}(\varphi)=
\tau_4(\varphi)+\hat\tau_4(\varphi).
\end{align}
Here, \(\tau_4(\varphi)\) denotes the $4$-tension field 
\begin{eqnarray*}
\tau_4(\varphi)&=&\bar\Delta^3\tau(\varphi)+\tr R^N(d\varphi(\cdot),\bar{\Delta}^2\tau(\varphi))d\varphi(\cdot)\\\nonumber
&&-\tr R^N(\bar\nabla_{(\cdot)}\bar{\Delta}\tau(\varphi),\tau(\varphi))d\varphi(\cdot)-\tr R^N(\bar\nabla_{(\cdot)}\tau(\varphi),\bar{\Delta}\tau(\varphi))d\varphi(\cdot) \,
\end{eqnarray*}
and the term \(\hat\tau_4(\varphi)\) is defined by
\begin{align}
\label{es-4-hat}
\hat{\tau}_4(\varphi)=-\frac{1}{2}\big(2\xi_1+2d^\ast\Omega_1+\bar\Delta\Omega_0+\tr R^N(d\varphi(\cdot),\Omega_0)d\varphi(\cdot)\big),
\end{align}
where we have used the following abbreviations
\begin{equation}
\label{variables-omega}
\begin{split}
\Omega_0&=R^N(d\varphi(e_i),d\varphi(e_j))(R^N(d\varphi(e_i),d\varphi(e_j))\tau(\varphi)), \\
\Omega_1(X)&=R^N(R^N(d\varphi(X),d\varphi(e_j))\tau(\varphi),\tau(\varphi))d\varphi(e_j),\\
\xi_1&=-(\nabla^N R^N)\big(d\varphi(e_j),R^N(d\varphi(e_i),d\varphi(e_j))\tau(\varphi),\tau(\varphi),d\varphi(e_i)\big ).
\end{split}
\end{equation}

We will prove the following versions of Theorems~\ref{Th-harmonic-everywhere} and \ref{main-theorem} for ES-\(4\)-harmonic maps:
\begin{Thm}
\label{theorem-es4-harmonic}
Let $\varphi \colon M \to N$ be an ES-$4$-harmonic map. If $\varphi$ is harmonic on an open subset $U$ of $M$, then $\varphi$ is harmonic everywhere.
\end{Thm}

\begin{Theorem}
\label{theorem-es4-main}
Let \(\varphi,\tilde\varphi\colon M\to N\) be two ES-$4$-harmonic maps. If they agree on an open subset $U$ of $M$, then they are identical.
\end{Theorem}

In order to prove the unique continuation Theorems \ref{theorem-es4-harmonic} and \ref{theorem-es4-main} we have to further differentiate the 
second and third term on the right hand side of \eqref{es-4-hat} as we need to write down their expressions in local coordinates.
We will then express all contributions in terms of the variables \(\{u_0,v_0=\nabla u_0,u_1,v_1=\nabla u_1,u_2\}\) which we previously
employed in the analysis of $4$-harmonic maps.

The last term 
of \eqref{es-4-hat} can easily be written in local coordinates
as we do not need to further differentiate 
it:
\begin{align*}
\big(\tr R^N(d\varphi(\cdot),\Omega_0)d\varphi(\cdot)\big)^\alpha
=&R^\alpha_{\beta\gamma\delta}R^\delta_{\vartheta\kappa\sigma}R^\vartheta_{\mu\nu\eta}
\langle d\varphi^\beta,d\varphi^\gamma\rangle\langle d\varphi^\kappa,d\varphi^\nu\rangle\langle d\varphi^\sigma,d\varphi^\eta\rangle u_0^\mu.
\end{align*}
This shows that the last
term of \eqref{es-4-hat} can be rewritten in terms of the desired
variables. Since all the terms in \eqref{es-4-hat} have a tensorial meaning, we can assume that $\{e_i\}_{i}$ is a geodesic frame field around an arbitrary point $p$ of $M$. Thus, at $p$, a computation shows that the second term of \eqref{es-4-hat} is:
\begin{align*}
-d^\ast\Omega_1=&
\underbrace{(\nabla^N R^N)\big(d\varphi(e_i),R^N(d\varphi(e_i),d\varphi(e_j))\tau(\varphi),\tau(\varphi),d\varphi(e_j)\big )}_{{=\xi_1}} \\
&+R^N\big ((\nabla_{d\varphi(e_i)}R^N)(d\varphi(e_i),d\varphi(e_j),\tau(\varphi)),\tau(\varphi)\big )d\varphi(e_j) \\
&+R^N\big(R^N(\tau(\varphi),d\varphi(e_j))\tau(\varphi),\tau(\varphi)\big )d\varphi(e_j) \\
&+R^N\big (R^N(d\varphi(e_i),{ \nabla}d\varphi(e_i,e_j))\tau(\varphi),\tau(\varphi)\big )d\varphi(e_j)\\
&+R^N\big (R^N(d\varphi(e_i),d\varphi(e_j))\bar\nabla_{e_i}\tau(\varphi),\tau(\varphi)\big )d\varphi(e_j)\\
&+R^N\big (R^N(d\varphi(e_i),d\varphi(e_j))\tau(\varphi),\bar\nabla_{e_i}\tau(\varphi)\big )d\varphi(e_j).
\end{align*}

Hence, again this output has a tensorial meaning and so it holds on $M$. In terms of local coordinates we have
\begin{align*}
2{\xi_1^\alpha}+2(d^\ast\Omega_1)^\alpha=&
-2\Big[R^\alpha_{\beta\gamma\delta}R^\gamma_{\mu\nu\vartheta;\epsilon}u_0^\mu u_0^\delta\langle d\varphi^\epsilon,d\varphi^\nu\rangle\langle d\varphi^\beta,d\varphi^\vartheta\rangle \\
&+R^\alpha_{\beta\gamma\delta}R^\gamma_{\mu\nu\vartheta}u_0^\mu u_0^\nu u_0^\delta\langle d\varphi^\beta,d\varphi^\vartheta\rangle 
+R^\alpha_{\beta\gamma\delta}R^\gamma_{\mu\nu\vartheta}u_0^\mu u_0^\delta \varphi^{\nu}_i \varphi^{\beta}_j g^{ik}g^{j\ell}(\nabla_k \varphi_\ell^\vartheta)\\
&+R^\alpha_{\beta\gamma\delta}R^\gamma_{\mu\nu\vartheta}\langle d\varphi^\nu,(\bar\nabla\tau(\varphi))^\mu\rangle
\langle d\varphi^\vartheta,d\varphi^\beta\rangle u_0^\delta \\
&+R^\alpha_{\beta\gamma\delta}R^\gamma_{\mu\nu\vartheta}\langle d\varphi^\nu,(\bar\nabla\tau(\varphi))^\delta\rangle\langle d\varphi^\vartheta,d\varphi^\beta\rangle u_0^\mu\Big]
\end{align*}
and so we deduce that the first two terms of the right hand side of \eqref{es-4-hat} can be rewritten in terms of the desired variables.
Here, we use a ``\;\(;\)\;'' to denote the covariant derivative of the curvature tensor.

Unfortunately, the third term on the right hand side of \eqref{es-4-hat} causes more technical difficulties. First we state the following lemma whose proof is standard and thus omitted.

\begin{Lem}\label{lem:T13-derivative}
Let $\varphi:M\to N$, $T\in C(T_3^1(N))$, and $\sigma_1,\sigma_2,\sigma_3,\sigma_4\in C(\varphi^{-1}TN)$. Define
$(\nabla_{\sigma_1}T)(\sigma_2,\sigma_3,\sigma_4)\in C(\varphi^{-1}TN)$ by
\[
\left((\nabla_{\sigma_1}T)(\sigma_2,\sigma_3,\sigma_4)\right)(p)=\left(\nabla_{\sigma_1(p)}T\right)(\sigma_2(p),\sigma_3(p),\sigma_4(p))\,, \quad \forall p\in M\,.
\]
Then, for $X\in C(TM)$, we have
\[
\begin{split}
\bar\nabla_{X}&\left((\nabla_{\sigma_1}T)(\sigma_2,\sigma_3,\sigma_4)\right)=
\left(\nabla_{\bar\nabla_{X}\sigma_1}T\right)(\sigma_2,\sigma_3,\sigma_4)+\left(\nabla^2 T\right)(d\varphi(X),\sigma_1,\sigma_2,\sigma_3,\sigma_4)\\
&+(\nabla_{\sigma_1}T)(\bar\nabla_{X}\sigma_2,\sigma_3,\sigma_4)
+(\nabla_{\sigma_1}T)(\sigma_2,\bar\nabla_{X}\sigma_3,\sigma_4)
+(\nabla_{\sigma_1}T)(\sigma_2,\sigma_3,\bar\nabla_{X}\sigma_4)\,.
\end{split}
\]
\end{Lem}
Now, we assume that for a given arbitrary point \(p\in M\), $\{e_i\}$ is a geodesic frame field around \(p\) and we perform the calculations at the point \(p\). We have
\begin{equation}
\label{Delta-omega0}
\begin{split}
\bar\Delta\Omega_0=&-\bar\nabla_{e_k}\Big [(\nabla_{d\varphi(e_k)}R^N)(d\varphi(e_i),d\varphi(e_j),R^N(d\varphi(e_i),d\varphi(e_j))\tau(\varphi)) \\
&+2R^N(\bar\nabla_{e_k}d\varphi(e_i),d\varphi(e_j))(R^N(d\varphi(e_i),d\varphi(e_j)\tau(\varphi))\\
&+R^N(d\varphi(e_i),d\varphi(e_j))\big ((\nabla_{d\varphi(e_k)}R^N)(d\varphi(e_i),d\varphi(e_j),\tau(\varphi))\big ) \\
&+2R^N(d\varphi(e_i),d\varphi(e_j))\big ( R^N(\bar\nabla_{e_k}d\varphi(e_i),d\varphi(e_j))\tau(\varphi)\big )\\
&+R^N(d\varphi(e_i),d\varphi(e_j))\big ( R^N(d\varphi(e_i),d\varphi(e_j))\bar\nabla_{e_k}\tau(\varphi)\big )
\Big ]\,.
\end{split}
\end{equation}

In order to express the terms in  \eqref{Delta-omega0} with respect to suitable variables we begin writing down the first addend. Using Lemma~\ref{lem:T13-derivative} we have
\begin{equation}
\label{Delta-omega0-first-term}
\begin{split}
\bar\nabla_{e_k}&\big[\big( \nabla_{d\varphi(e_k)} R^N\big)\big(d\varphi(e_i),d\varphi(e_j),R^N(d\varphi(e_i),d\varphi(e_j))\tau(\varphi)\big) \big]=\\
&\big( \nabla_{\tau(\varphi)} R^N\big)\big(d\varphi(e_i),d\varphi(e_j),R^N(d\varphi(e_i),d\varphi(e_j))\tau(\varphi)\big)\\
+&\big(\nabla^2 R^N\big)\big(d\varphi(e_k),d\varphi(e_k),d\varphi(e_i),d\varphi(e_j),R^N(d\varphi(e_i),d\varphi(e_j))\tau(\varphi)\big)\\
+&2\big(\nabla_{d\varphi(e_k)}R^N\big)\big(\bar\nabla_{e_k}d\varphi(e_i),d\varphi(e_j),R^N(d\varphi(e_i),d\varphi(e_j))\tau(\varphi)\big)\\
+&
\big(\nabla_{d\varphi(e_k)}R^N\big)\big(d\varphi(e_i),d\varphi(e_j),\big(\nabla_{d\varphi(e_k)}R^N\big)\big(d\varphi(e_i),d\varphi(e_j),\tau(\varphi) \big)\big)\\
+&2 \big(\nabla_{d\varphi(e_k)}R^N\big)\big(d\varphi(e_i),d\varphi(e_j),R^N(\bar\nabla_{e_k}d\varphi(e_i),d\varphi(e_j))\tau(\varphi)\big)\\
+&\big(\nabla_{d\varphi(e_k)}R^N\big)\big(d\varphi(e_i),d\varphi(e_j),R^N(d\varphi(e_i),d\varphi(e_j))\bar\nabla_{e_k}\tau(\varphi)\big)\,.
\end{split}
\end{equation}
Since, at $p$, $\bar\nabla_{e_k}d\varphi(e_i)=\nabla d\varphi(e_k,e_i)$ we can conclude that all terms in \eqref{Delta-omega0-first-term} have a tensorial character.
As a consequence we can replace the geodesic frame field $\{e_i\}$ by the local coordinates frame field $\{\partial/\partial x^i\}$ and, since all terms are linear in $\tau(\varphi)$ or $\nabla_{({\partial}/{\partial x^i})}\tau(\varphi)$, they can be estimated by $u_0^{\alpha}$ and 
$\partial u_0^{\alpha}/\partial x^i$. To obtain the correct estimates of the other addends in 
\eqref{Delta-omega0} it is enough to show that the terms
\[
{\rm (i)}\;\; \bar \nabla_{e_k}\bar \nabla_{e_k}d\varphi (e_i)\,,\quad {\rm (ii)}\;\; \bar\nabla_{e_k}\bar\nabla_{e_k}\tau(\varphi)
\]
have a tensorial character. For (i), applying the Weitzenb\"ock  formula (see, for example, \cite{MR703510} or \cite[Proposition 1.34]{MR1391729}), we obtain
\[
\bar\nabla_{e_k}\bar \nabla_{e_k}d\varphi (e_i)=R^N(d\varphi (e_k),d\varphi (e_i))d\varphi(e_k)+d\varphi(\operatorname{Ric}^M(e_i))+\bar\nabla_{e_i}\tau(\varphi)
\]
which shows that $\bar\nabla_{e_k}\bar \nabla_{e_k}d\varphi (e_i)$ has indeed a tensorial character. While, for (ii), we have 
\[
\bar\nabla_{e_k}\bar\nabla_{e_k}\tau(\varphi)=-\bar\Delta\tau(\varphi)\,.
\]
At the end all addends in \eqref{Delta-omega0} have a tensorial character and 
replacing the geodesic frame field $\{e_i\}$ by the local coordinates frame field $\{\partial/\partial x^i\}$ they can be estimated by $u_0^{\alpha}$,  
$\partial u_0^{\alpha}/\partial x^i$ and $u_1^{\alpha}$.

At this point we have realized that all terms on the right hand side of \eqref{es-4-hat} 
can be expressed in terms of the required variables.

\begin{proof}[Proof of Theorem~\ref{theorem-es4-harmonic} and Theorem~\ref{theorem-es4-main}] In order to prove Theorem~\ref{theorem-es4-harmonic} we define the 
vector-valued function
\[
u=\left ( \begin{array}{l}
u_0\\
v_0\\
u_1\\
v_1\\
u_2
\end{array} \right )
\]
and by the same analysis as in the proof of Theorem \ref{Th-harmonic-everywhere} we show using Aronszajn's Theorem that \(u=0\).

In order to prove Theorem \ref{theorem-es4-main} we define the vector-valued function
\[
u=\left ( \begin{array}{l}
\varphi\\
d\varphi\\
u_0\\
v_0\\
u_1\\
v_1\\
u_2
\end{array} \right )
\]
and \(\tilde u\) will be defined accordingly. Employing the same strategy as in the proof
of Theorem \ref{main-theorem}, now it is easy to complete the proof of Theorem \ref{theorem-es4-main}.
\end{proof}

\par\medskip
\textbf{Acknowledgements:}
The first author gratefully acknowledges the support of the Austrian Science Fund (FWF) 
through the project P30749-N35 ``Geometric variational problems from string theory''. The second and the last author  were supported by Fondazione di Sardegna (project STAGE) and Regione Autonoma della Sardegna (Project KASBA).

\end{document}